\documentclass[a4paper,11pt,DIV=11,
abstract=on
]{scrartcl}
\usepackage{algorithm, booktabs}
\usepackage{mathtools}
\mathtoolsset{showonlyrefs}
\usepackage{amsmath}
\usepackage{amssymb}
\usepackage{array}
\usepackage{amsthm}
\usepackage{subcaption}
\usepackage{mathdots}
\usepackage{mathrsfs}
\usepackage{algorithm}
\usepackage[noend]{algpseudocode}
\usepackage{graphicx}
\usepackage{color}
\usepackage{listings, enumitem}
\usepackage{tikz}
\usepackage{todonotes}
\usepackage{hyperref}
\usepackage{float}
\usepackage[markup=underlined]{changes}
\definechangesauthor[color=orange]{MG}

\newcommand{\weakly}{\rightharpoonup}
\newcommand{\weaklystar}{\overset{\ast}{\weakly}}
\newcommand{\N}{\ensuremath{\mathbb{N}}}

\newcommand{\eps}{\varepsilon}

\renewcommand{\S}{\ensuremath{\mathbb{S}}}

\newcommand{\R}{\ensuremath{\mathbb{R}}}

\newcommand{\dx}{\,\mathrm{d}}
\newcommand{\tT}{\mathrm{T}}

\DeclareMathOperator{\BV}{BV}
\DeclareMathOperator{\TGV}{TGV}

\def\3{\ss}

\newmuskip\pFqmuskip

\newcommand*\pFq[6][8]{
  \begingroup 
  \pFqmuskip=#1mu\relax
  \begingroup\lccode`\~=`\,
  \lowercase{\endgroup\let~}\pFqcomma
  {}_{#2}F_{#3}{\left(\genfrac..{0pt}{}{#4}{#5};#6\right)}%
  \endgroup
}
\newcommand*\pRegFq[6][8]{
  \begingroup 
  \pFqmuskip=#1mu\relax
  \begingroup\lccode`\~=`\,
  \lowercase{\endgroup\let~}\pFqcomma
  {}_{#2}\tilde{F}_{#3}{\left(\genfrac..{0pt}{}{#4}{#5};#6\right)}%
  \endgroup
}

\newcommand{\pFqcomma}{\mskip\pFqmuskip}

\DeclareMathOperator*{\trace}{trace}

\DeclareMathOperator*{\SO}{SO}

\DeclareMathOperator*{\argmin}{argmin}

\DeclareMathOperator{\TV}{TV}

\newtheorem{theorem}{Theorem}[section]
\newtheorem{lemma}[theorem]{Lemma}
\newtheorem{remark}[theorem]{Remark}

\begin{document}
\title{An Image Registration Model in
\\
Electron Backscatter Diffraction}

\author{
Manuel Gr\"af\footnotemark[1]
	\and
Sebastian Neumayer\footnotemark[1]
\and
Ralf Hielscher\footnotemark[2]
	\and
Gabriele Steidl\footnotemark[1]
\and
Moritz Liesegang\footnotemark[3]
\and
Tilman Beck\footnotemark[3]
	}

\maketitle
\renewcommand{\thefootnote}{\fnsymbol{footnote}}
\footnotetext[1]{Institute of Mathematics,
	TU Berlin,
	Germany,
	\{graef, neumayer, steidl\}@math.tu-berlin.de.}

	\footnotetext[2]{Institute of Mathematics,
	TU Chemnitz,
	Germany,
	ralf.heilscher@mathematik.tu-chemnitz.de.}

	\footnotetext[3]{Dept.\ of Mechanical and Process Engineering,
	TU Kaiserslautern,
	Germany,
	\{liesega, beck\}@mv.uni-kl.de.}

\begin{abstract}
Variational methods were successfully applied for registration of gray and RGB-valued image sequences. A common assumption in these models is that pixel-values do not change
under transformations.
Nowadays, modern image acquisition techniques such as electron backscatter tomography (EBSD), which is used in material sciences, can capture images with values in nonlinear spaces.
Here, the image values belong to the quotient space $\SO(3)/ \mathcal S$ 
of the special orthogonal group modulo the discrete symmetry group of the crystal. 
For such data, the assumption that pixel-values remain unchanged under transformations appears to be no longer valid.
Hence, we propose a variational model for the registration of $\SO(3)/\mathcal S$-valued image sequences, taking the dependence of pixel-values on the transformation into account.
More precisely, the data is transformed according to the rotation part in the polar decomposition of the Jacobian of the transformation.
To model non-smooth transformations without obtaining so-called staircasing effects,
we propose to use a total generalized variation like prior.
Then, we prove existence of a minimizer for our model and explain how it can be discretized and minimized by a primal-dual algorithm.
Numerical examples illustrate the performance of our method.
\end{abstract}

\section{Introduction}
%
Variational methods for estimating the displacement between image frames
go back to Horn and Schunck \cite{HS81}. 
Meanwhile, there exists a vast number of refinements and extensions of their approach and we refer to \cite{BPS14,WBTN2006} 
for an overview. 
In particular, models with priors containing higher order derivatives of the displacement field
were  successfully used, e.g., in \cite{HK14,RBP14,TPCB08,YSM07,YSS07}.
In material science, such models were applied for the strain analysis in materials \cite{BaBeEiFiSchSt19,HS2021,HWSSD13},  
where they appear to be more sensitive to abrupt changes in the displacement field than correlation based methods 
used in state-of-the-art software packages such as \cite{BlAdAn15,VEDDAC,VIC}.
While optical flow models with a linearized data term as described above are mainly convex, 
nonconvex models have to be solved in registration \cite{Mod2004,BookMo09}, large deformation diffeomorphic metric mapping \cite{BMTY2005} or metamorphosis \cite{younes2010shapes}.

Nowadays, modern image acquisition techniques can not only produce gray-valued and RGB images,
but also images with values in nonlinear spaces.
Typical examples are diffusion tensor magnetic resonance tomography (DT-MRI), 
where the image values are symmetric
positive definite matrices, and electron backscatter diffraction (EBSD), 
where the image values are from a certain quotient space 
of the rotation group $\SO(3)$.
Recently, the metamorphosis approach of Trouv\'e and Younes \cite{TY2005b} 
in its path discrete form introduced by Berkels et al.~\cite{BER15}, see also \cite{RW2015},
was generalized to manifold-valued images in \cite{NPS2018} with a mathematically sound theory
for Hadamard manifolds in \cite{ENR2020}.
However, the usual ,,gray-value constancy assumption'' from videos does not carry over to the
manifold-valued setting. 
Instead, for image sequences whose values contain directional information 
the data has to also be transformed spatially.
For DT-MRI images, the appropriate handling of orientations was addressed in, 
e.g., \cite{APBG2001,AG2000,CMMWY2006,Yeo2009}.  

In this paper, we focus on sequences of EBSD images, which appear in the microstructural analysis of crystalline materials. 
For every pixel of an EBSD image the phase and the crystal orientation is
measured based on electron diffraction on the crystal lattice of the
specimen \cite{AdWrKu93, KuWrAdDi93}.
The orientations are given by a rotation in $\SO(3)$ modulo 
the finite symmetry group $\mathcal S$ of the crystal.
Polycrystalline materials usually consist of clearly separated regions with similar orientations, so-called grains.
With EBSD, it is possible to analyze the grain structure of the specimen, e.g., the size and distribution of grains, the shape and characteristics of grain boundaries or orientations. 
Based on the microstructural analysis, engineers can
draw conclusions to macrostructural mechanical or functional behavior of the material, e.g., its ductility, fatigue or electrical properties.
EBSD is also used to investigate temporary or permanent microstructural changes, caused by an external influence, such as an applied force \cite{Singh.2018,Wilkinson.2012}. Reconstruction of data back to the initial state by optical flow allows to determine small changes of the microstructure or to quantify local deformations and orientation changes in a region of irregular distributed deformations that can not be detected by common analysis methods.
For more information on EBSD, we refer to \cite{MS06} and for the practical visualization and analysis of EBSD data to  the software package MTEX \cite{MTEX,BHS11}. The segmentation of EBSD data was studied in \cite{BFPS17} and for a statistical model to observed texture evolution of fatigued metal films see \cite{PNK2020}. 

We propose a variational model for estimating the displacement field between EBSD images
that consists of a special data term and prior: 
\begin{itemize}
\item In accordance with the finite strain reorientation strategy \cite{APBG2001}, the data term takes the rotation part of the transformations' Jacobian into account.
Moreover, it relies on the geometric distance in the quotient manifold $\SO(3)/\mathcal S$
and uses the quaternion representation of $\SO(3)$.
\item  The prior (or regularization term) is based on the total generalized variation (TGV)
introduced for image restoration of gray-valued images by Bredies et al.~\cite{BKP10}, see also \cite{SS08,SST11} for the discrete setting.
Later, the concept was generalized for reconstructing tensor-valued images by Valkonen et al.~\cite{VaBrKn13}.
This regularizer allows for sharp transitions in the displacement field without the staircasing effect known, e.g., from total variation regularization. 
\end{itemize}

Let us roughly illustrate how the rotation part in the transformation influences the pixel-values.
Given two images $I_1\colon \Omega \to \mathrm{SO(3)}/\mathcal S$ and $I_2\colon \Omega_1 \to \mathrm{SO(3)}/\mathcal S$, 
we want to find the transformation $\varphi\colon\Omega \to \Omega_1$ between them.
A first idea would be to determine $\varphi$ so that
\begin{equation} \label{eq:naiveTransform}
  I_2 \circ \varphi(x) = I_1(x), \qquad x \in \Omega.
\end{equation} 
In other words, the value $I_1(x) \in \mathrm{SO(3)}/\mathcal S$ at position $x \in \Omega$ 
is taken to the position $\varphi(x)$.
For gray-valued images this corresponds exactly to the gray-value constancy assumption. 
However, for orientation data this naive approach does not reflect the physical conditions, 
since a transformation by a rotation $R \in \mathrm{SO(3)}$ also changes the orientation of the EBSD data, 
i.e., we have that
\begin{equation}   \label{eq:rigidRotation}
    I_2(R x) = R I_1(x), \qquad x \in \Omega.
\end{equation} 
In general, the transformation $\varphi$ is not rigid, 
so that the change in orientation can vary locally. 
For determining the rotation field from the transformation $\varphi$, 
we will use the polar decomposition of $\nabla \varphi$ from continuum mechanics. Then the refined version of \eqref{eq:rigidRotation} for general transformations $\varphi$ is given by \eqref{eq:EBSDTransform} and introduced in Section~\ref{sec:basic_model}.

Figure~\ref{fig:rotating_square} illustrates the difference between  
our variational model with data term given by the straightforward approach \eqref{eq:naiveTransform} and the rotation aware model \eqref{eq:EBSDTransform}. The top row of Figure~\ref{fig:rotating_square} shows two piecewise constant images $I_1$, $I_2$. The square-shaped region of image $I_1$ and its orientation data is rotated in image $I_2$ by 30 degrees. The difference between the colors of the squares in image $I_1$ and image $I_2$ reflect the change of the data induced by the transformation. The orientation data of the background (black) is unchanged. In the bottom row of Figure~\ref{fig:rotating_square} we observe that both models are able to recover the shape of the rotated square. Of course, the orientation data does not change in the reconstruction of $I_2$ by model \eqref{eq:naiveTransform}, cf. bottom left of Figure~\ref{fig:rotating_square}. As intended the rotation aware model \eqref{eq:EBSDTransform} changes the orientation in accordance with the transformation $\varphi$, cf. bottom right of Figure~\ref{fig:rotating_square}. Moreover, the rotation aware model \eqref{eq:EBSDTransform} is able to resolve quite accurate the rigid transformation of the square, even though we are solving for general transformations. In contrast, the naive model \eqref{eq:naiveTransform} introduces large deformation artifacts near the boundaries inside the square.
\begin{figure}
  \begin{center}
    \includegraphics[width=0.48\textwidth]{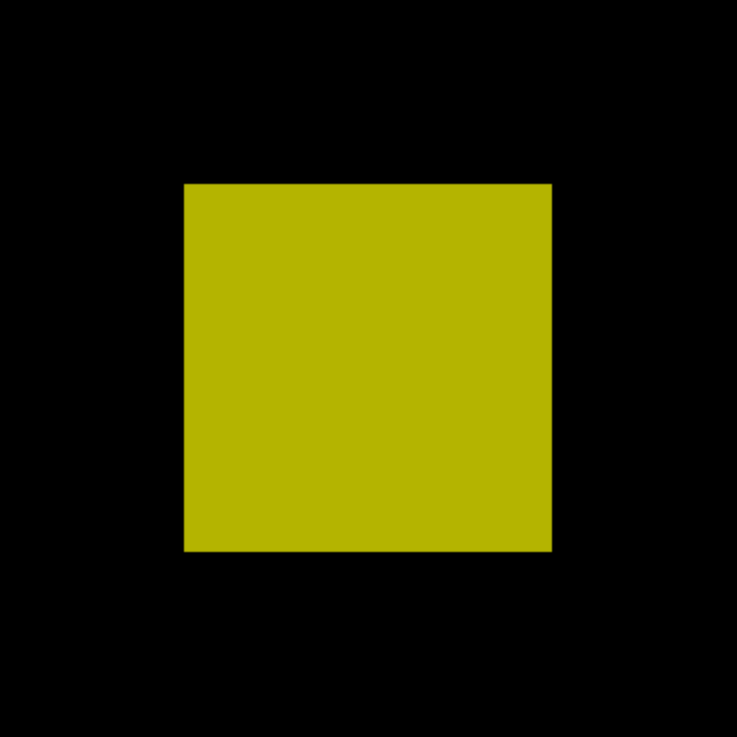}
    \includegraphics[width=0.48\textwidth]{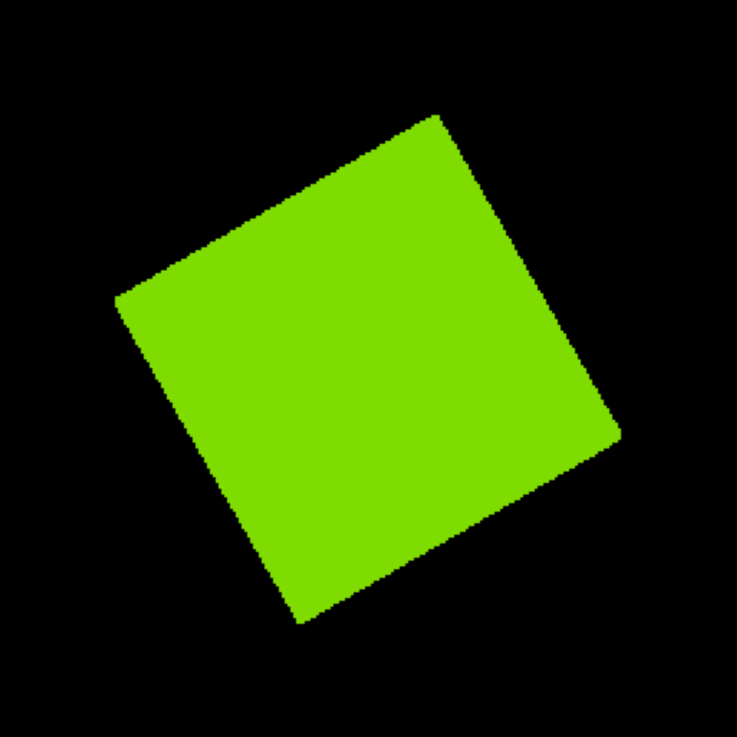}
    \includegraphics[width=0.48\textwidth]{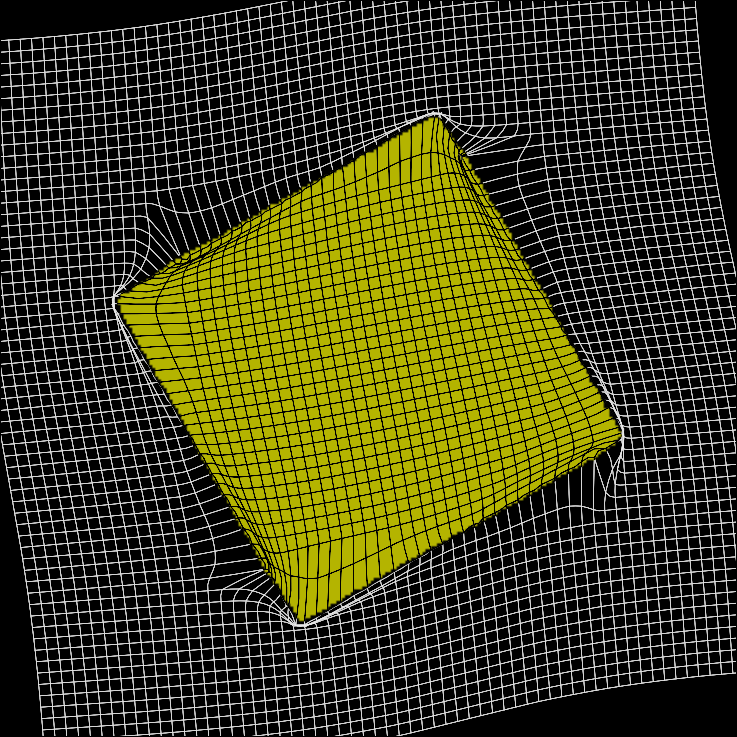}
    \includegraphics[width=0.48\textwidth]{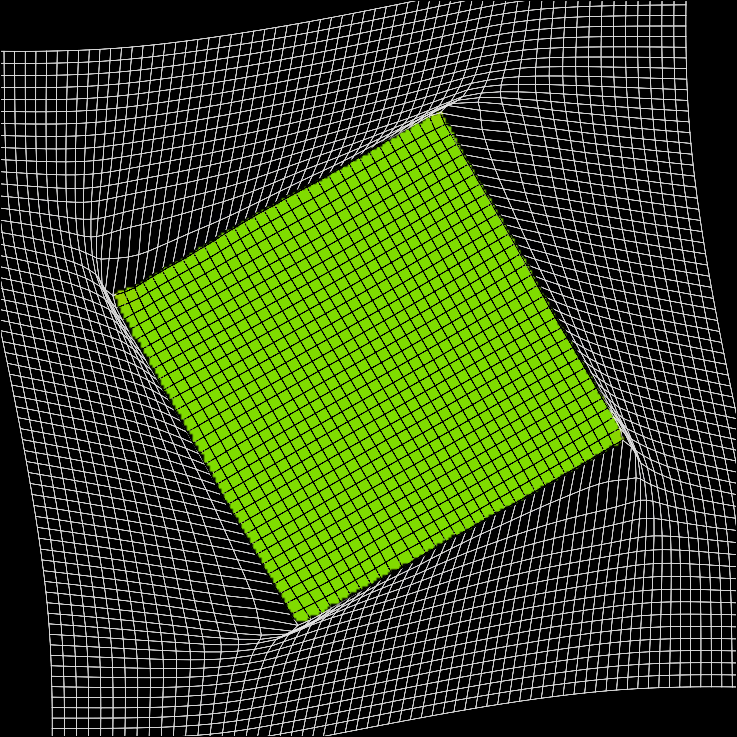}
  \end{center}
  \caption{Top row: Image $I_1$ (left) and image $I_2$ (right) of $\SO(3)$-valued data visualized in RGB space. Bottom row: Reconstruction of $I_2$ overlaid with a grid visualizing $\varphi$ computed with the naive model \eqref{eq:naiveTransform} by $I_1 \circ \varphi^{-1}$ (left) and our new rotation aware model \eqref{eq:EBSDTransform} by $\mathrm R (\nabla \varphi) I_1 \circ \varphi^{-1}$ (right). }
  \label{fig:rotating_square}
\end{figure}

This paper is organized as follows:
In Section~\ref{sec:pre}, we provide an overview on functions 
of bounded variation and matrix-valued Radon measures. 
Then, in Section~\ref{sec:c-model}, we introduce our continuous variational model to determine the optical flow
between $\SO(3)/\mathcal S$-valued images and prove the existence of minimizers.
In Section~\ref{sec:d-model} we discretize the proposed variational model by sampling bilinear approximations at different scales.
Here, we make use of the quaternion representation of $\mathrm{SO}(3)$, which is introduced at the beginning of the section.
A primal-dual optimization algorithm is proposed in Section~\ref{sec:alg}.
We emphasize that our discretization and optimization strategy 
is particularly suited for using parallel computing devices.
The proposed algorithms are implemented for GPU devices using Python together with the CUDA toolkit.
In Section~\ref{sec:numerics}, we demonstrate the performance of our algorithm for synthetic as well as real-world data.
Finally, we draw conclusions and indicate directions of future research in Section~\ref{sec:conclusions}.
%
\section{Preliminaries} \label{sec:pre}
%
To establish our variational model, several technical preliminaries are necessary.
Readers who are familiar with the topic may move immediately to the next section.
We mainly follow the lines of \cite{AFP00,VaBrKn13}, where we stick to unsymmetrized tensors.
By $\mathcal T^k(\mathbb R^d)$, $k \in \mathbb N$, we denote
the set of all $k$-tensors on the vector space $\mathbb R^d$, i.e., for $k \ge 1$
the $k$-linear mappings $A\colon \mathbb R^d \times \cdots \times \mathbb R^d \to \mathbb R$. By convention, 
a $0$-tensor is real number, $T^1(\mathbb R^d)$ 
is the vector space $\mathbb R^d$ and $\mathcal T^2(\mathbb R^d)$ 
is the vector space of $d \times d$ matrices. 
Taking the standard basis 
$e_i \in \mathbb R^d$, $i=1,\dots,d$, any tensor $A \in \mathcal T^k(\mathbb R^d)$ 
is uniquely determined by its coefficients 
$A_{i_1,\dots,i_k} \coloneqq A(e_{i_1},\dots, e_{i_k})$, $i_l \in \{1,\dots, d\}$, $l=1,\dots,k$.
With the inner product and the associated \emph{Frobenius norm} 
\[
  \langle A, B \rangle = \sum_{i \in \{1,\dots,d\}^k} A_{i} B_{i}, \quad 
	\| A \|_{F} \coloneqq \sqrt{\langle A, A \rangle}\qquad A, B \in \mathcal T^k(\mathbb R^d),
\]
the space $\mathcal T^k(\mathbb R^d)$ becomes a Hilbert space.
For a bounded domain $\Omega \subset \mathbb R^d$ with Lipschitz boundary,  
we define the $L^p$-space of $p$-integrable $k$-tensor fields 
$u\colon \Omega \to \mathcal T^k(\mathbb R^d)$ by
\[
L^p(\Omega, \mathcal T^k(\mathbb R^d)) \coloneqq 
\bigl\{ u\colon\Omega \to \mathcal T^k(\mathbb R^d) \;:\; \|u\|_{p} < \infty \bigr\} , 
\qquad p \in [1,\infty],
\]
where the $p$-norm is defined as
\[
  \| u \|_p \coloneqq 
  \begin{cases} 
    \big( \int_\Omega \|u\|_F^p \dx x \big)^\frac1p, & p \in [1,\infty),\\
    \mathrm{ess}\sup_{x \in \Omega} \| u(x) \|_F, & p = \infty.
  \end{cases}
\]
A tensor field $u\colon\Omega \to \mathcal T^k(\mathbb R^d)$ is differentiable 
if all coordinate functions $u_{i_1,\dots,i_k}\colon\Omega \to \mathbb R$ are differentiable. 
The class of $l$-times continuously differentiable $k$-tensor fields is denoted by $C^l(\Omega, \mathcal T^k(\mathbb R^d))$, $k\in \mathbb N$, and the subspace of compactly supported $k$-tensor fields by $C^l_c(\Omega, \mathcal T^k(\mathbb R^d))$.
Further, the closure of $C^l_c(\Omega, \mathcal T^k(\mathbb R^d))$ with respect to the $C^l$-norm is denoted by $C^l_0(\Omega, \mathcal T^k(\mathbb R^d))$.
For $u \in L^1(\Omega, \mathcal T^k(\mathbb R^d))$, the \emph{distributional gradient} 
$\nabla u \in C^1_c(\Omega, \mathcal T^{k+1}(\mathbb R^d))^*$ is defined by
\begin{equation}   \label{eq:Gradient}
  \nabla u(\psi) \coloneqq \int_\Omega \langle u, \mathrm{div}\, \psi \rangle \dx x, \qquad \psi \in C^1_c(\Omega, \mathcal T^{k+1}(\mathbb R^d)),
\end{equation}
where for $\psi \in C_c^1(\Omega, \mathcal T^k(\mathbb R^d))$, $k\ge 1$, the \emph{divergence} $\mathrm{div}\, \psi \in C(\Omega, \mathcal T^{k-1}(\mathbb R^d))$ is given as
\[
  (\mathrm{div}\, \psi)_{i_1,\dots,i_{k-1}}(x) \coloneqq \sum_{i=1}^d \frac{\partial \psi_{i,i_1,\dots,i_{k-1}}}{\partial x_{i}}(x), \qquad i_l \in \{1,\dots,d\}, \quad l=1,\dots,k-1.
\]
Then, the \emph{total variation} of $u \in L^1(\Omega, \mathcal T^{k}(\mathbb R^d))$ is given by
\[
  \TV(u) \coloneqq \sup \bigl\{\nabla u(\psi) \;:\; \psi \in C^1_c(\Omega, \mathcal T^{k+1}(\mathbb R^d)),\, \|\psi\|_\infty \le 1 \bigr\}.
\]
The \emph{space of tensor fields of bounded variation} is defined as
\begin{equation}   \label{eq:BV}
\BV(\Omega, \mathcal T^k(\mathbb R^d)) \coloneqq \bigl\{ u \in L^1(\Omega, \mathcal T^k(\mathbb R^d)) \;:\; \TV(u) < \infty \bigr\}.
\end{equation}
A function $\mu\colon\mathcal B(\Omega) \to \mathcal T^k(\mathbb R^d)$ on the Borel $\sigma$-algebra $\mathcal B(\Omega)$ is called a 
\emph{$k$-tensor-valued Radon-measure} if every coordinate function $\mu_{i_1,\dots,i_d}\colon\mathcal B(\Omega) \to \mathbb R$ is a Radon measure. 
We denote by $\mathcal M(\Omega, \mathcal T^k(\mathbb R^d))$  the \emph{space of $k$-tensor-valued finite Radon-measures}. 
By the Riesz--Markov--Kakutani representation theorem, it holds $C_0(\Omega, \mathcal T^k(\mathbb R^d))^* \cong \mathcal M(\Omega, \mathcal T^k(\mathbb R^d))$.
This allows us to equip the space $\mathcal M(\Omega, \mathcal T^k(\mathbb R^d))$ with the corresponding \emph{weak* convergence}.
Since for $u \in \mathrm BV(\Omega, \mathcal T^k(\mathbb R^d))$ we have
\[
  |\nabla u(\psi)| \le \TV(u) \|\psi\|_\infty, \qquad \psi \in C^1_c(\Omega, \mathcal T^{k+1}(\mathbb R^d)),
\]
and since $C_c^1(\Omega, \mathcal T^{k+1}(\mathbb R^d))$ 
is dense and continuously embedded in $C_0(\Omega, \mathcal T^{k+1}(\mathbb R^d))$, 
the gradient $\nabla u$ can be uniquely extended to a functional on $C_0(\Omega, \mathcal T^{k+1}(\mathbb R^d))$ using the Hahn--Banach theorem.
Hence, we can associate to $\nabla u$ a unique measure $Du \in \mathcal M(\Omega, \mathcal T^{k+1}(\mathbb R^d))$ such that 
\begin{equation}   \label{eq:TVmeasure}
  \nabla u(\psi) = 
  \sum_{i \in \{1,\dots,d\}^{k+1}} \int_\Omega \psi_i \dx Du_{i}, \qquad \psi \in C^1_c(\Omega, \mathcal T^{k+1}(\mathbb R^d)).
\end{equation}
In the rest of the paper, we only require $D u$ and reuse $\nabla u$ for the density of $D u$ with respect to the Lebesgue measure.
For a measure $\mu \in \mathcal M(\Omega, \mathcal T^k(\mathbb R^d))$, we define the \emph{total variation norm} by
\[
  \vert \mu \vert(\Omega) 
	\coloneqq 
	\sup \Bigl\{ \sum_{i \in \{1,\dots,d\}^{k}} \int_\Omega  \psi_i \dx \mu_i \;:\; 
	\psi \in C_c(\Omega, \mathcal T^k(\mathbb R^d)), \quad \|\psi\|_\infty \le 1 \Bigr\}
\]
and the (second order) \emph{total generalized variation} of a tensor field $u \in \BV(\Omega, \mathcal T^k(\mathbb R^d))$ by
\begin{equation} \label{eq:tgv}
  \mathrm{TGV}_{\alpha} (u) \coloneqq \inf_{w \in \BV(\Omega, \mathcal T^{k+1}(\mathbb R^d))} 
	\alpha_1 \vert Du - w \lambda \vert(\Omega)  + \alpha_2 \vert Dw \vert(\Omega) ,\quad \alpha = (\alpha_1,\alpha_2)> 0,
\end{equation}
where we identify $w \in \BV(\Omega, \mathcal T^{k+1}(\mathbb R^d))$ with the canonical measure
\[
(w\lambda)(\mathcal A) \coloneqq \int_{\mathcal A} w \dx x, \qquad \mathcal A \in \mathcal B(\Omega).
\]
The space $\BV(\Omega, \mathcal T^k(\mathbb R^d))$ becomes a Banach space with the norm
\[
 \| u \|_{\mathrm BV} \coloneqq \|u\|_1 + \TV(u).  
\]
A sequence $u_n$ \emph{converges weakly*} in $\BV(\Omega, \mathcal T^k(\R^d))$ if $u_n \to u$ 
strongly in $L^1(\Omega, \mathcal T^k(\R^d))$ and $D u_n \weaklystar Du$ in the sense of measures.
Equivalently, we can require $\sup_n \TV(u_n) < \infty$ instead of $D u_n \weaklystar Du$.
Recall that the space $C^\infty(\overline{\Omega}, \mathcal T^k(\R^d))$ is dense w.r.t.\ weak* convergence.
Further, any sequence 
$u_n \in \BV(\Omega, \mathcal T^k(\R^d))$ with $\sup_{n \in \N} \Vert u_n \Vert_{\BV}< \infty$ 
admits a weakly* convergent subsequence.

The following theorem provides a generalization of the polar decomposition to matrix-valued Radon measures, 
cf. \cite[Sec.~4]{RoRo68}. Let $\mathrm{SPD}(d)$ denote the cone of symmetric, positive semi-definite matrices.

\begin{theorem}  \label{the:PolarDecompositionMeasure} 
  Let $\mu = \big(\mu_{i,j}\big)_{i,j=1}^d \in \mathcal M(\Omega, \mathcal T^2(\mathbb R^d))$ and $\sigma$ 
	be a Radon measure such that the component-wise measures $\mu_{i,j}$ are absolutely continuous w.r.t.\ $\sigma$, i.e., 
  \[
   \mu(\mathcal A) = \int_{\mathcal A} M \dx \sigma, \qquad \mathcal A \in \mathcal B(\Omega),
  \]
  for a measurable matrix-valued function $M \coloneqq (M_{i,j})_{i,j=1}^d$.
  Then $|\mu|$ defined by
  \[
    |\mu|(\mathcal A) \coloneqq \int_{\mathcal A} V \dx \sigma, 
		\qquad 
		\mathcal A \in \mathcal B(\Omega),  \qquad V \coloneqq \big( M^\tT  M\big)^\frac12, 
  \]
  is a Radon measure with density $V\colon \Omega \rightarrow \mathrm{SPD}(d)$.
	The measure $|\mu|$ does not depend on the choice of $\sigma$.
  Further, there exists a matrix-valued function $R\colon \Omega\to \mathrm O(d)$ 
	such that all component-functions $R_{i,j}\colon \Omega \to \mathbb R$, $i,j=1,\dots,d$, are measurable and 
   \[
    \mu_{i,j}(\mathcal A) = \sum_{k=1}^d \int_{\mathcal A} R_{i,k} \dx |\mu|_{k,j}, \qquad \mathcal A \in \mathcal B(\Omega).
   \]
   If additionally $\det V \ne 0$ $\sigma$-a.e., then $R$ 
	is uniquely determined $\sigma$-a.e.
\end{theorem}

\section{Continuous Image Registration Model} \label{sec:c-model}
Let $\Omega, \Omega_1 \subset \R^2$ be bounded
domains with Lipschitz boundary.
Further, assume that 
$I_1 \in C (\overline{\Omega}, \mathrm{SO(3)}/\mathcal{S})$ and $I_2 \in C (\overline{\Omega}_1, \mathrm{SO(3)}/\mathcal{S})$ 
are two given images.
In the following, we discuss various different variational models for such data, 
where the optimization domain is 
$\BV(\Omega, \overline{\Omega}_1)$, i.e., functions in $\BV(\Omega, \R^2)$ 
whose range is restricted to $\overline{\Omega}_1$.
Naturally, the constraint $\varphi(\Omega) \subset \overline{\Omega}_1$ has to be understood in an a.e.\ sense.
More precisely, there has to be a representative such that
$\varphi(\Omega) \subset \overline{\Omega}_1$ holds.
Note that $\BV(\Omega, \overline{\Omega}_1)$ is a closed subset of $\BV(\Omega, \R^2)$
w.r.t.\ $L^1(\Omega, \mathbb R^2)$-convergence and hence also w.r.t.\ weak* convergence.
Further, $\BV(\Omega, \overline{\Omega}_1)$ is convex if $\overline{\Omega}_1$ is convex. 
From now on, we always denote the distributional gradient of
$\varphi \in \BV(\Omega, \overline{\Omega}_1)$ by $D \varphi \in \mathcal M(\Omega, \mathcal T^{2}(\mathbb R^2))$ and 
the \emph{density of its absolutely continuous part} by $\nabla \varphi$, 
i.e., 
$D \varphi = \nabla \varphi \lambda + (D \varphi)^\perp$.
Note that $\nabla \varphi$ coincides with the weak gradient of $\varphi$ if it exists.

We are looking for an appropriate variational model to determine the transformation $\varphi$
in sequences of EBSD images of the form
\begin{equation}\label{eq:general}
\inf_{\varphi \in \BV(\Omega,  \overline{\Omega}_1)} \mathcal D\bigl(\varphi;I_1,I_2\bigr) 
+ \mathcal R(\varphi),
\end{equation}
where the  data term $\mathcal D\bigl(\varphi; I_1,I_2\bigr)$ takes rotation of the data during the transformation
into account and the regularizer $\mathcal R(\varphi)$ makes the the problem well-posed.

In the following subsection, we propose a basic model for which the existence of a minimizer is not ensured.
However, the subsequent subsections give modifications of the setting
making the problem well-posed in the sense that a minimizer exists.

\subsection{Basic Model}\label{sec:basic_model}
Recall that the pixel values are equivalence classes 
$[R]  \coloneqq \{ R S \;:\; S \in \mathcal S\} \in \SO(3) / \mathcal S$.
In the data term, we use for  $[R_1], [R_2] \in \SO(3) / \mathcal S$ the distance  
\begin{equation} \label{eq:distance}
\dx_{\mathrm{SO(3)}/\mathcal{S}} ([R_1],[R_2])
\coloneqq 
\min_{S \in \mathcal S}
\dx_{\mathrm{SO(3)}} ( R_1 S,  R_2)
\end{equation}
induced by the geodesic distance 
$$ 
\dx_{ \mathrm{SO(3)} }(R_1, R_2) \coloneqq \sqrt{2} \,  \mathrm{arccos} \left( \tfrac12 (\trace(R_1^\tT R_2) - 1) \right)
$$
on $\mathrm{SO(3)}$, see \cite{St1964}.
Note that the distance \eqref{eq:distance} is rotation invariant.
Later we use the representation of the elements of $\mathrm{SO(3)}$ as quaternions 
and adapt the distance accordingly.
To take the rotation of the data into account, we use the polar decomposition 
of the density $\nabla \varphi \colon \Omega \to \mathcal T^2(\R^2)$.
If  $\det(\nabla \varphi(x)) > 0$, 
there exists a unique \emph{polar decomposition} 
\begin{equation}   \label{eq:polarDecomposition}
  \nabla \varphi(x) = R_\varphi(x) V_\varphi(x), \quad x \in \Omega,
\end{equation}
where $R_\varphi (x) \in \mathrm{SO}(2)$ and $V_\varphi(x) \in \mathrm{SPD}(2)$.
More precisely, if $\nabla \varphi = U \Sigma \tilde U^\tT$ denotes the singular value decomposition of 
$\nabla \varphi(x)$, then we have
\begin{equation}   \label{eq:polarDecomposition_svd}
R_\varphi = \nabla \varphi \left(\nabla \varphi^\tT \nabla \varphi \right)^{-\frac12} = U \tilde U^\tT, \quad 
V_\varphi = \tilde U \Sigma \tilde U^\tT.
\end{equation}
Note that Theorem \ref{the:PolarDecompositionMeasure} provides a polar decomposition for more general measures.
However, using such an approach, it is not obvious how to ensure uniqueness or even lower semi-continuity of the corresponding functional based on some regularizer as it is done below.
EBSD techniques acquire only two-dimensional images of three-dimensional material probes.
Using the two-dimensional transformation $\varphi\colon \Omega \rightarrow \bar \Omega_1$, 
we are only able to catch changes in orientation by rotations around the axis orthogonal to the image plane, 
so that we arrive at the following transformation model of EBSD data
\begin{equation}   \label{eq:EBSDTransform}
   I_2\circ\varphi = \mathrm{R}(\nabla \varphi)  I_1, 
	\qquad 
	\mathrm{R}(\nabla \varphi) \coloneqq\begin{pmatrix} 
	R_\varphi & 0 \\ 0 & 1 \end{pmatrix}.
\end{equation}
In summary, a natural choice for the data term appears to be
\begin{align} \label{eq:data}
\mathcal D \bigl(\varphi;I_1,I_2\bigr) 
&\coloneqq 
\int_{\Omega} 
\dx_{\mathrm{SO(3)}/\mathcal{S}} \bigl( 
\mathrm{R}(\nabla \varphi) I_1, I_2 \circ \varphi
\bigr) \dx x = \int_{\Omega} 
\dx_{\mathrm{SO(3)}/\mathcal{S}} \left( 
I_1, \mathrm{R}(\nabla \varphi)^\tT I_2 \circ \varphi
\right) \dx x ,\quad\,\,
\end{align}
where the last equality holds true due to the rotation invariance of the distance.

Now, the regularizer $\mathcal R$ must be constructed such that it ensures 
a unique polar decomposition in order to make the whole functional well-defined.
Further, the regularizer should control the $\BV$ norm to guarantee its coercivity.
A good candidate seems to be
\begin{equation} \label{eq:orig} 
  \mathcal R(\varphi) \coloneqq \TGV_\alpha(\varphi - \mathrm{Id}) + \beta \int_\Omega f(\det \nabla \varphi) \dx x,
	\qquad 
	 \beta > 0,
\end{equation}
where
\begin{equation} \label{def:f}
f(x) \coloneqq 
\left\{                         
\begin{array}{ll}
x^{-1} + x& \mathrm{if} \; x>0,\\
+\infty   &\mathrm{otherwise}. 
\end{array}
\right.
\end{equation}
The second summand in the regularizer ensures 
that the polar decomposition exists for a.e.\ $x \in \Omega$ as soon as the energy is finite.
However, this regularizer is not weakly* lower semi-continuous due to the second summand, cf.~\cite[p.~182]{Dep12}.
An alternative would be to use the modified regularizer 
\begin{equation} \label{eq:relax}                 
  \mathcal R_{\mathrm{relax}} (\varphi) 
	\coloneqq 
	\inf_{\substack{\varphi_n \to \varphi \text{ in } L^1(\Omega, \R^2) \\ \varphi_n 
	\in W^{1,2}(\Omega, \R^2)}} \bigl\{ \liminf_{n \rightarrow \infty} \mathcal R(\varphi_n)\bigr\}, 
	\qquad 
	\varphi \in \BV(\Omega, \R^2).
\end{equation}
Note, that  $\mathcal R_{\mathrm{relax}}$ is the relaxation 
of the functional $\mathcal R|_{W^{1,2}(\Omega, \mathbb R^2)}$ 
to $\BV(\Omega, \R^2)$, see also \cite{AD94,AD92,FM97,FM92,KR10} for similar approaches.
By Lemma \ref{lem:xxx} in the appendix, we have that indeed 
$\mathcal R_{\mathrm{relax}}(\varphi) = \mathcal R(\varphi)$ for $\varphi \in W^{1,2}(\Omega,\R^2)$.
Here  $W^{1,p}(\Omega,\R^2)$, $p \in [1,\infty)$ denotes the Banach space of functions in $L^p(\Omega, \R^2)$
having weak first order derivatives in $L^p(\Omega, \R^2)$ equipped with the norm
$$
\|\varphi\|_{W^{1,p}} \coloneqq \left( \|\varphi\|_{p}^p  +  \|\tfrac{\partial}{\partial x}  \varphi\|_p^p  
+ 
\|\tfrac{\partial}{\partial y}  \varphi\|_p^p \right)^\frac{1}{p}.
$$
Further, we prove the following theorem in Appendix \ref{app:A}.

\begin{theorem}\label{thm:basic}
Let the data term $\mathcal D\bigl(\varphi;I_1,I_2 \bigr)$ be weakly* lower semi-continuous. 
Then the problem
\begin{equation}\label{eq:var_problem_orig}
\inf_{\varphi \in \BV(\Omega,  \overline{\Omega}_1)} 
\mathcal D\bigl(\varphi;I_1,I_2 \bigr)
+ 
\mathcal R_{\mathrm{relax}} (\varphi) 
\end{equation}
has a minimizer.
\end{theorem}

Unfortunately, our data term
$\mathcal D\bigl(\varphi;I_1,I_2 \bigr)$ 
in \eqref{eq:data}  
is  not even quasiconvex and establishing weak* lower semi-continuity appears to be challenging.
In particular, all lower semi-continuity results for integral functionals that 
we are aware of require either quasiconvexity or some stronger 
notion of convergence than the weak* convergence.
Therefore, we investigate two modifications of the general setting 
that ensure the existence of a minimizer by exploiting stronger modes of convergence.

\subsection{Smoothing of \texorpdfstring{$D\varphi$}{D Phi}}\label{sec:well-posed_modifications_1}
We choose some mollifier $\rho_\eps \in C_c^\infty(\R^2)$ with (potentially small) smoothing parameter $\eps >0$.
Its \emph{convolution} with the Radon measure $D\varphi \in \mathcal M (\Omega, \mathcal T^2(\R^2) )$  
is the function defined by 
$$
\rho_\eps \ast D\varphi = \int_{\Omega} \rho_\eps(\cdot-y) \dx D\varphi(y).
$$
Then, we  replace $\nabla \varphi$ by $\rho_\eps \ast D\varphi$ in \eqref{eq:data} and \eqref{eq:orig}
and consider the variational problem
\begin{equation}\label{eq:var_problem1}
\inf_{\varphi \in \BV(\Omega,  \overline{\Omega}_1)} 
\mathcal D\bigl(\varphi;I_1,I_2\bigr) + \TGV_\alpha(\varphi - \text{Id}) + \beta\int_{\Omega} f\bigl(\det(\rho_\eps \ast D \varphi)\bigr) \dx x
\end{equation}
with
\begin{equation} \label{eq:data_moll}
\mathcal D \bigl(\varphi;I_1,I_2\bigr) 
\coloneqq
\int_{\Omega} 
\dx_{\mathrm{SO(3)}/\mathcal{S}} \bigl( 
\mathrm{R}(\rho_\eps \ast D\varphi) I_1, I_2\circ \varphi
\bigr) \dx x.
\end{equation}
To show existence of minimizers, we use the following compactness result, see \cite[Lem.~B.2]{HVW15}.
\begin{lemma}\label{lem:comp}
Let $\rho_\eps  \in C_c^\infty(\R^2)$ and suppose $\mu_n \weaklystar \mu \in \mathcal M(\Omega;\R^2)$.
Then $\rho_\eps \ast \mu_n \to \rho_\eps \ast \mu$ in $L^\infty(\R^2)$.
\end{lemma}
The next theorem establishes the desired existence of a minimizer.
\begin{theorem}\label{thm:exist_smooth}
The variational problem \eqref{eq:var_problem1} admits a minimizer.
\end{theorem}
\begin{proof}
Let $\varphi_n \in \BV(\Omega, \overline{\Omega}_1)$ be a minimizing sequence.
Then, it holds $\sup_{n}\|\varphi_n \|_1 < \infty$ as $\varphi_n(\Omega) \subset \overline{\Omega}_1$ 
is bounded.
Moreover, \cite[Cor.~3.13]{BH14} and the triangle inequality imply for any 
$\varphi \in \BV(\Omega, \overline{\Omega}_1)$ that
\begin{align}\label{eq:est_TV}
    \TV(\varphi) \le C \bigl(\Vert \varphi \Vert_1 + \TGV_\alpha(\varphi)\bigr) \leq C \TGV_\alpha(\varphi - \text{Id}) + C.
\end{align}
Hence, we get
\[\sup_{n}\|\varphi_n \|_1 + \TV (\varphi_n) < \infty,\]
and there exists a subsequence, again denoted with $(\varphi_{n})_n$, 
converging strongly (and also a.e.) 
to some $\varphi$ in $L^1(\Omega, \overline{\Omega}_1)$.
Moreover, $D\varphi_n$ converges weakly* to $D\varphi$ 
and Lemma~\ref{lem:comp} implies that $\rho_\eps \ast D \varphi_n \to \rho_\eps \ast D \varphi \in L^\infty(\R^2)$.
Clearly, this implies existence of a subsequence of $(\varphi_n)_n$, 
again denoted with $(\varphi_n)_n$, 
such that $\rho_\eps \ast D \varphi_n \to \rho_\eps \ast D \varphi$ a.e.

Recall that $\TGV_\alpha$ is lower semi-continuous w.r.t.\ $L^1(\Omega, \mathbb R^2)$-convergence, see \cite[Proof of Prop.~3.5]{BKP10}.
Further, we can use the continuity of $f$ to conclude $f(\rho_\eps \ast D \varphi_n) \to f(\rho_\eps \ast D \varphi)$ a.e.\ and hence the lemma of Fatou implies
\[\int_{\Omega} f\bigl(\det(\rho_\eps \ast D \varphi)\bigr) \dx x 
\leq 
\liminf_{n \to \infty} \int_{\Omega} f\bigl(\det(\rho_\eps \ast D \varphi_n)\bigr) \dx x.\]
Note that the continuity of $I_2$ implies $I_2\circ \varphi_n \to I_2 \circ \varphi$ a.e.
Then, the continuity of $\dx_{\mathrm{SO(3)}/\mathcal S}$ 
and of $\mathrm R$ on the set of invertible matrices gives 
\[
\dx_{\mathrm{SO(3)}/\mathcal S}\bigl(\mathrm R(\rho_\eps \ast D \varphi_n)  I_1, I_2\circ \varphi_n\bigr) 
\to 
\dx_{\mathrm{SO(3)}/G}\bigl(\mathrm R(\rho_\eps \ast D \varphi)  I_1, I_2\circ \varphi\bigr) \quad \text{a.e.},
\]
which again together with the lemma of Fatou implies lower semi-continuity of the data term $\mathcal D$.
Consequently, $\varphi$ is a minimizer of our functional \eqref{eq:var_problem1}.
\end{proof}

\subsection{Higher Regularity of \texorpdfstring{$\varphi$}{Phi}}\label{sec:well-posed_modifications_2}
Next, we propose a model that is based on higher regularity of the transformation $\varphi$.
More precisely, 
we restrict the transformations to the space
\[\BV^2(\Omega,  \overline{\Omega}_1) 
\coloneqq 
\bigl\{\varphi \in W^{1,1}(\Omega,\overline{\Omega}_1): 
\nabla \varphi \in \BV(\Omega, T^2(\R^2))\bigr\},\]
see \cite[Sec.~9.8]{SGGH09} for more details.
A sequence $\varphi_n$ converges weakly* in $\BV^2(\Omega, \overline{\Omega}_1)$ 
if $\varphi_n \to \varphi$ strongly in $W^{1,1}(\Omega, \overline{\Omega}_1)$ 
and $D^2 \varphi_n \weaklystar D^2 \varphi$.
Equivalently, we can require 
$\sup_n \TV(\nabla \varphi_n) < \infty$ instead of 
$D^2 \varphi_n \weaklystar D^2 \varphi$.
Further, it holds that any sequence $(\varphi_n)_n$ with 
$\sup_n \Vert \varphi_n \Vert_{W^{1,1}} + \TV(\nabla \varphi_n) < \infty$ 
admits a weakly* convergent subsequence.

Now, we consider the variational problem
\begin{equation}\label{eq:var_problem2}
\inf_{\varphi \in \BV^2(\Omega,  \overline{\Omega}_1)} 
\mathcal D \bigl(\varphi;I_1,I_2\bigr)  + \alpha \TV(\nabla \varphi - I_2) + \beta \int_{\Omega} f(\det \nabla \varphi) \dx x,
\end{equation}
where as in \eqref{eq:data},
$$
\mathcal D \bigl(\varphi;I_1,I_2\bigr) = \int_{\Omega} 
\dx_{\mathrm{SO(3)}/\mathcal{S}} \bigl( 
\mathrm{R}(\nabla \varphi) I_1, I_2 \circ \varphi \bigr) \dx x.
$$
Again, the second regularizer in \eqref{eq:var_problem2} ensures that $\mathrm{R}(\nabla \varphi)$ is well-defined for 
a.e.\ $x \in \Omega$ as soon as the energy is finite.
In order to establish existence of a minimizer, we need the following lemma.
\begin{lemma}\label{lem:EstNorm}
There exists a constant $C>0$ such that for every $\varphi \in \BV^2(\Omega, \R^2)$ and $w \in \BV(\Omega, T^2(\R^2))$ it holds
\[\Vert \varphi \Vert_{W^{1,1}} \le C \bigl(\Vert \varphi \Vert_1 + \| \nabla \varphi - w \|_{1} + \TV(w)\bigr).\]
In particular, it holds $\Vert \varphi \Vert_{W^{1,1}} \le C (\Vert \varphi \Vert_1 + \TV(\nabla \varphi))$.
\end{lemma}

\begin{proof}
Assume in contrast that there exits no constant $C \in \R$ such that the inequality holds.
Then, there exist sequences 
$(\varphi_n)_n \in \BV^2(\Omega, \R^2)$ 
and 
$(w_n)_n \in \BV(\Omega, T^2(\R^2))$ with
\[\Vert \varphi_n \Vert_{W^{1,1}} = 1 \quad \text{and} \quad \frac{1}{n} \geq\Vert \varphi_n \Vert_1 + \Vert \nabla \varphi_n -w_n\Vert_1 + TV(w_n).\]
From $\Vert \varphi_n \Vert_{W^{1,1}} = 1$ we infer that $(w_n)_n$ is bounded in $\BV(\Omega, T^2(\R^2))$ and admits a weakly* convergent subsequence with limit $w \in \BV(\Omega, T^2(\R^2))$.
Hence, it holds $\nabla \varphi_n \to w$ in $L^1(\Omega,T^2(\R^2))$.
As $\varphi_n \to 0$ in $L^1(\Omega,\R^2)$, we further get $\varphi_n \to 0$ in $W^{1,1}(\Omega,\R^2)$.
However, this contradicts our assumption $\Vert \varphi_n \Vert_{W^{1,1}} = 1$.
\end{proof}

Now, we can prove the actual existence result.

\begin{theorem}
The variational problem \eqref{eq:var_problem2} admits a minimizer.
\end{theorem}

\begin{proof}
Let $(\varphi_n)_n \in \BV^2(\Omega, \overline{\Omega}_1)$ be a minimizing sequence.
Then, it holds $\sup_{n}\|\varphi_n \|_1 < \infty$ 
as $\varphi_n(\Omega) \subset \overline{\Omega}_1$ is bounded.

Using Lemma~\ref{lem:EstNorm}, we get
\[\sup_{n}\|\varphi_n \|_{W^{1,1}} + \TV (\nabla \varphi_n) 
\le 
\sup_{n} C \bigl(\Vert \varphi_n \Vert_1 + \TV(\nabla \varphi_n)\bigr) < \infty,\]
and there exists a subsequence, again denoted with $(\varphi_{n})_n$, converging strongly to some 
$\varphi$ in $W^{1,1}(\Omega, \overline{\Omega}_1)$.
Clearly, this also implies the existence of a subsequence for which $(\nabla \varphi_n)_n$ is a.e.\ convergent.
The lower semi-continuity of the regularizer and the data term is shown similar as in Theorem~\ref{thm:exist_smooth}.
Consequently, $\varphi$ is a minimizer of our functional \eqref{eq:var_problem2}.
\end{proof}

\section{Discrete Image Registration Model} \label{sec:d-model}
In this section, we establish a discrete variant
of our model.
We make use of the quaternion representation of $\mathrm{SO}(3)$, which is recalled in the next subsection.

\subsection{Quaternion Representation of \texorpdfstring{$\SO(3)$}{SO(3)}}
There are several ways to represent the elements of the rotation group
$\SO(3)$, for example by real-valued $3 \times 3$ matrices or by three Euler angles.
In this paper, we focus on the representation by quaternions of unit length.
Compared to the matrix representation only four components are needed. Moreover, the calculations for quaternions are more convenient than for the representation by Euler angle.
For further information the reader may consult, e.g., \cite{Gr12}.

Quaternions are elements of the form 
$q = (s,v)^\tT \in \mathbb R \times \mathbb R^3$,
which form a 4-dimensional real vector space and together with the 
Hamiltonian multiplication
\[
  q_1 \odot q_2 \coloneqq 
	\begin{pmatrix}
	s_{1}s_{2} - v_{1}^{\tT}v_{2}\\
	s_{1} v_{2}  + s_{2} v_{1} - v_{1} \times v_{2} 
	\end{pmatrix},
\]
where $\times$ denotes the vector product, also  a division algebra.
Note that the Hamiltonian multiplication is associative, but not commutative.
The conjugate of  
$q = (s,v)^\tT \in \mathbb R \times \mathbb R^3$ is given by 
$
  \overline{q} = (s, -v)^\tT 
$
and its norm or length by 
$$
|q| \coloneqq \left( s^2 + \|v\|_2^2 \right)^\frac12.
$$
Let $\mathbb S^d$ denote the unit sphere in $\mathbb R^{d+1}$.
For the connection between quaternions and rotations, we restrict our attention to quaternions of unit lengths
$q  \in \mathbb S^3$, which can be uniquely represented by a vector $r \in \mathbb S^2$ and an angle
$\theta \in [0,2\pi)$ as
$$
q(r,\theta) \coloneqq \left( \cos \left(\tfrac{\theta}{2} \right), r \sin \left(\tfrac{\theta}{2}\right)\right)^\tT. 
$$
It is easy to check that the multiplication of two unit quaternions is again a unit quaternion.
On the other hand,
every rotation $R(r,\theta) \in \mathrm{SO(3)}$ 
is determined by a rotation axis $r \in \mathbb S^2$ 
with rotation angle $\theta \in \R$ and acts on a point $p \in \mathbb R^3$ by
$$
R(r,\theta) p = r(r^\tT p) + \cos(\theta) \left( (r \times p) \times  r\right) + \sin(\theta) (r \times p) \in \R^3.
$$ 
This is equivalent to multiplying $p$ with the matrix
\begin{equation} \label{rot_matrix}
R(r,\theta) = 
\begin{pmatrix}
(1-c) r_1^2+ c      & (1-c)r_1r_2 - r_3 s & (1-c) r_1r_3 + r_2 s\\
(1-c)r_1r_2 + r_3 s & (1-c) r_2^2+ c      &(1-c)r_2r_3 - r_1 s\\
(1-c) r_1r_3 - r_2 s&(1-c)r_2r_3 + r_1 s  &(1-c)r_3^2 + c
\end{pmatrix},
\end{equation}
where
$
\mathrm{c} \coloneqq \cos(\theta)$ and $\mathrm{s} \coloneqq \sin(\theta)
$. 
Note that $R(r,2\pi -\theta) = R(-r,\theta)$.
Rotations can be identified with unit quaternions by
\[
  R(r,\theta) \, \widehat{=}  \, q(r,\theta) \qquad r \in \mathbb S^2, \, \theta \in \R.
\]
Then, we have 
$R(r,\theta)^{\tT} \, \widehat{=} \, \overline{q(r,\theta)}$
and the homeomorphism between the rotation group and the multiplicative group of the quaternion algebra
\[
  R(r_{1},\theta_{1}) R(r_{2},\theta_{2})  \, \widehat{=} \, q(r_1,\theta_1) \odot q(r_2,\theta_2). 
\]
Since the same rotation is generated by 
$
-q = ( \cos (\tfrac{2\pi - \theta}{2}), \sin (\tfrac{2\pi - \theta}{2}) \, (-r))^\tT \in \mathbb S^3
$,
we see that $\mathrm{SO}(3) \cong \mathbb S^3/\{\pm 1\}$.
From these relations we infer that the geodesic distances on the rotation group $\mathrm{SO(3)}$ and $\mathbb S^3/\{\pm 1\}$
are related via
\[
  \mathrm{d}_{\mathrm{SO(3)}}\bigl(R(r_1,\theta_1),R(r_2,\theta_2)\bigr) 
= 2\sqrt{2}\, \mathrm{d}_{\mathbb S^3/\{\pm 1\}}\bigl((s_1,v_1),(s_2,v_2)\bigr) = 2\sqrt{2} \arccos | s_1 s_2 + v_1^\tT  v_2 |.
\]

To rewrite the data term \eqref{eq:data} with respect to quaternions, 
we have to determine the quaternion representation of the matrix
$\mathrm{R}(\nabla \varphi)$ in \eqref{eq:EBSDTransform}.
By \eqref{rot_matrix}, we see immediately that 
$r = \mathrm{e}_3 \coloneqq (0,0,1)^\tT$ 
so that 
\begin{equation} \label{quater_1}
\mathrm{R} \left(\nabla \varphi \right) 
\, \widehat{=} \, 
q\left(\mathrm{e}_3, \theta_\varphi \right) =
\bigl(\cos \bigl(\tfrac{\theta_\varphi}{2} \bigr),0,0,\sin \bigl(\tfrac{\theta_\varphi}{2} \bigr) \bigr)^\tT
\end{equation}
and 
$$
\mathrm{R}(\nabla \varphi)  = 
\left(
\begin{array}{rrr}
\cos (\theta_\varphi)&- \sin (\theta_\varphi)&0\\
\sin (\theta_\varphi)&\cos (\theta_\varphi)&0\\
0&0&1
\end{array}
\right).
$$
Further, we obtain from the polar decomposition that
$$
\nabla \varphi \coloneqq \left(
\begin{array}{rr}
a_{1,1}&a_{1,2}\\
a_{1,2}&a_{2,2}
\end{array}
\right) 
=
\left(
\begin{array}{rr}
\cos (\theta_\varphi)&-\sin (\theta_\varphi)\\
\sin (\theta_\varphi)&\cos (\theta_\varphi)
\end{array}
\right) 
\left(
\begin{array}{rr}
v_{1,1}&v_{1,2}\\
v_{1,2}&v_{2,2}
\end{array}
\right) , 
$$
where the latter matrix is in $\mathrm{SPD}(2)$.
Now, if $\det(\nabla \varphi) > 0$ a.e., straightforward computation implies
\begin{equation}   \label{eq:polarAngle2d}
  \theta_\varphi \coloneqq 
  \begin{cases}
    \arctan \left( \frac{a_{2,1} - a_{1,2}}{ a_{1,1} + a_{2,2}} \right), & a_{1,1} + a_{2,2} > 0,\\
    \pi - \arctan \left( \frac{a_{2,1} - a_{1,2}}{ a_{1,1} + a_{2,2}} \right), & a_{1,1} + a_{2,2} < 0.
    
    \end{cases}
\end{equation}
In summary, the data term \eqref{eq:data} with respect to quaternions can be written up to the factor $2\sqrt{2}$ and with the agreement that $\{\pm 1\}$ is incorporated in the symmetry group $\mathcal S$ as
\begin{equation} \label{eq:data_quaternion}
\mathcal D \bigl(\varphi;I_1,I_2\bigr) 
\coloneqq 
\int_{\Omega} 
\dx_{\mathbb S^3/\mathcal{S}} \bigl( 
I_1, \overline{q\left(\mathrm{e}_3, \theta_\varphi \right)} \odot I_2 \circ \varphi
\bigr)  \dx x, 
\end{equation}
where $\theta_\varphi$ is determined by \eqref{eq:polarAngle2d}.

\subsection{Discretization}
%
We discretize the image registration model 
\begin{equation} \label{eq:var_problem_orig_old} 
\inf_{\varphi \in \BV(\Omega,  \overline{\Omega}_1)} 
\mathcal D\bigl(\varphi;I_1,I_2 \bigr)
+ 
\mathcal R (\varphi) 
\end{equation}
with data term \eqref{eq:data_quaternion} and regularizer \eqref{eq:orig}
by using finite dimensional, bilinear approximations of the involved functions 
and by discretizing the corresponding integrals using equidistant samples. 
Henceforth, we assume that working on a discrete grid already provides a smoothing
in the sense of Subsection~\ref{sec:well-posed_modifications_1} with an appropriately small
$\varepsilon$.
We have also implemented the ``higher regularity of $\varphi$'' approach from 
Subsection~\ref{sec:well-posed_modifications_2} and will show results
in the numerical part. 
However, we only briefly comment on the  modified discretization for this approach
in Remark~\ref{rem:discrete_modifications}. 

Using the notation
$$
\varphi = \text{Id} + u,  \quad \nabla \varphi = \mathbf{I}_2 + \nabla u, \quad \theta_\varphi = \theta_u
$$
with the $2 \times 2$ identity matrix $\mathbf{I}_2$
and recalling the TGV definition \eqref{eq:tgv}, we aim to minimize a discrete version of
\begin{equation}   \label{eq:var_problem_discrete}
  \begin{aligned}
  E(u,w) \coloneqq  \int_{\Omega}  &  \dx_{\S^3/\mathcal S}
	\bigl(
	I_1, \overline{q\left(\mathrm{e}_3,\theta_u\right)} \odot I_2 \circ (\text{Id} + u) \bigr) + \alpha_1 \| \nabla u - w\|_F \\
	&+ \alpha_2 \| \nabla w \|_F  + \beta f\bigl(\det(\mathbf{I}_2 + \nabla u)\bigr)  \dx x.
  \end{aligned}
\end{equation}

Let the domains 
$\Omega\coloneqq (0, a) \times (0, b)$ 
and 
$
\Omega_1\coloneqq (0, a_1) \times (0, b_1)$
be rectangles 
with sides having integer lengths $a,b,a_1,b_1 \in \mathbb N$.
We assume that all occurring functions can be approximated 
by interpolation at prescribed sampling points. 
More precisely, we define the bilinear interpolation basis function $B\colon \R^2 \to \R$ by
\[
  B(x) = 
  \begin{cases}
    (1-|x_1|)(1-|x_2|), & |x_1|,|x_2| \le 1,\\
    0, & \text{else}.  
  \end{cases}
\]
Then, for given EBSD data 
$I_1(x_{i,j}) \in \mathbb S^3 /\mathcal S$ 
and
$I_2(x_{i,j}) \in \mathbb S^3 /\mathcal S$
sampled at the grid points 
$x_{i,j} \coloneqq (i+\tfrac12, j+\tfrac12)$, 
the corresponding functions 
$I_1\colon\Omega \to \mathbb S^3 /\mathcal S$ and $I_2\colon\Omega_1 \to \mathbb S^3 / \mathcal S$ 
are given as follows. 
Let $q(r_{1,i,j}, \theta_{1,i,j}) \in I_{1}(x_{i,j})$ and $q(r_{2,i,j}, \theta_{2,i,j}) \in I_{2}(x_{i,j})$ be unit quaternions of minimal angles $\theta_{1,i,j}, \theta_{2,i,j} \in [0, \pi]$ with corresponding vectors $r_{1,i,j}, r_{2,i,j} \in \S^2$, respectively.
Note that for almost all elements of $\S^3 / \mathcal S$ there is exactly one quaternion of minimal angle.
Using these representatives, we interpolate the given EBSD data by defining
\begin{align} \label{eq:images}
  I_1(x)  & \coloneqq \mathrm{Proj}_{\S^3 /\mathcal S} 
	\biggl( \sum_{i=0}^{a} \sum_{j=0}^{b} q(r_{1,i,j}, \theta_{1,i,j}) B(x - x_{i,j}) \biggr), \qquad x \in \Omega,\\
	I_2(x)  & \coloneqq \mathrm{Proj}_{\S^3 /\mathcal S} 
	\biggl( \sum_{i=0}^{a_1} \sum_{j=0}^{b_1} q(r_{2,i,j}, \theta_{2,i,j}) B(x - x_{i,j}) \biggr), \qquad x \in \Omega_1,
\end{align}
where $\mathrm{Proj}_{\S^3 / \mathcal S}\colon \R^4 \setminus \{ 0 \} \to \S^3 / \mathcal S$ 
denotes the orthogonal projection onto $\S^3 / \mathcal S$ defined by
\[
    \mathrm{Proj}_{\S^3 / \mathcal S} (q) \coloneqq \bigl[ q/\|q\|_2 \bigr]_{\mathcal S}, \qquad q \in \R^4 \setminus \{ 0 \}.
\]
We want to mention that this interpolation approach does not respect the proper topology of the quotient space $\S^3 / \mathcal S$, since in general the distance in $\SO(3)$ of the involved representatives could be larger than the distances in $\SO(3)/\mathcal S$ of the corresponding equivalence classes, cf. \eqref{eq:distance}. However, the given formula is easy to implement and leads to reasonable results, as seen by the numerical experiments in Section~\ref{sec:numerics}.

For our multilevel approach, we shall approximate the displacement field $u \in \BV(\Omega, \overline{\Omega}_1)$ and the tensor field $w \in \BV(\Omega, T^2(\R^2))$ at different scales $s$. 
More precisely, $s>0$ is a scaling factor such that $a / s, b / s \in \N$. Then we define $u_s\colon \Omega \to \overline \Omega_1$, the bilinear approximation of $u$ at scale $s$, by the expansion
\[
  u_s(x) \coloneqq \sum_{i=0}^{a / s} \sum_{j=0}^{b / s} u_{s,i,j} B\big(x / s - (i,j)^\tT\big),\qquad x \in \Omega,
\]
where $u_{s,i,j} \in \mathbb R^2$, $i=0,\dots,a/s$, $j=0,\dots,b/s$, are the expansion coefficients. 
Further, the piecewise constant approximation of $w$ at scale $s$ is defined by $w_s\colon\Omega \to T^2(\R^2)$ using the expansion
\[
 w_s(x) \coloneqq  \sum_{i=0}^{ a / s - 1} \sum_{j=0}^{b / s - 1} w_{s,i,j} 1_{\Omega_{s,i,j}}(x), \qquad  1_{\Omega_{s,i,j}}(x)  \coloneqq 
 \begin{cases}
   1, & x \in \Omega_{s,i,j},\\
   0, &\text{else}.
 \end{cases}
\]
where $w_{s,i,j} \in T^2(\R^2)$, $i=0,\dots,a/s-1$, $j=0,\dots,b/s-1$, are the values on the subdomains
\[
  \Omega_{s,i,j} \coloneqq (s i, s(i+1)) \times (s j, s (j+1)), \qquad i=0,\dots, a/s - 1, \quad j=0,\dots, b/s  - 1.
\]
Note that the subdomains $\Omega_{s,i,j}$ have side length $s$, which can be made arbitrarily small. In particular, decreasing scales $s$ lead to finer resolutions of the approximations $u_s$ and $w_s$ of $u$ and $w$, respectively. Moreover, the coarsest resolution is given by the scale $s_{\max} := \mathrm{gcd}(a,b)$, the greatest common divisor of $a$ and $b$.

Finally, the energy \eqref{eq:var_problem_discrete} is approximated at scale $s$ by the sampled energy
\begin{equation}
  \label{eq:energyDiscretized}
  E_s(u_s, w_s) \coloneqq s^2 \sum_{i=0}^{a/s - 1} \sum_{j=0}^{b/s - 1} D_{s,i,j}(u_s) + \alpha_1  \TV^1_{s,i,j}(u_s,w_s) + \alpha_2 \TV^2_{s,i,j}(w_s) + \beta F_{s,i,j}(u_s),
\end{equation}
where $D_{s,i,j}$, $\TV^1_{s, i,j}$, $\TV^2_{s, i,j}$ and $F_{s,i,j}$ are the following discretizations of the corresponding integrals in \eqref{eq:var_problem_discrete} on the subdomains $\Omega_{s,i,j}$.
For a given sampling size $m \in \N$, we use the sampling points 
\[
  x_{s,i,j,k,l} \coloneqq \Big(s\big(i+\tfrac{2k+1}{2m}\big), s\big(j + \tfrac{2l+1}{2m}\big)\Big) \in \Omega_{s,i,j}
\]
and define
\[
  \begin{aligned}
  & D_{s,i,j} (u_s)  
	\coloneqq 
	\frac{1}{m^{2}} \sum_{k,l=0}^{m-1} 
  \dx_{\S^3 / \mathcal S}\Big(I_1(x_{s,i,j,k,l}),
	\overline{q(\mathrm{e}_3, \theta_u(x_{s,i,j,k,l}))} \odot I_2\bigl(x_{s,i,j,k,l} + u_s(x_{s,i,j,k,l})\bigr)\Big), \\
  & \TV^1_{s,i,j}(u_s,w_s) 
	\coloneqq 
	\frac{1}{m^{2}} \sum_{k,l=0}^{m-1}  \| \nabla u_s(x_{s,i,j,k,l}) - w_{s,i,j}\|_F, \\
  &\TV^2_{s,i,j} (w_s) 
	\coloneqq
  \frac{1}{s}
  \begin{cases}
  \big( \| w_{s,i+1,j} - w_{s,i,j} \|_F^2 + \| w_{s,i,j+1} - w_{s,i,j} \|_F^2 \big)^\frac12, & i < a/s - 1,\; j < b/s -1,\\
  \| w_{s,i+1,j} - w_{s,i,j} \|_F, & i < a/s-1,\; j = b/s -1,\\
  \| w_{s,i,j+1} - w_{s,i,j} \|_F, & i =  a/s - 1,\; j <   b/s - 1,\\
  \end{cases}\\
  & F_{s,i,j} (u_s)  \coloneqq 
	\frac{1}{m^{2}} \sum_{k,l=0}^{m-1} f\bigl(\det(\mathbf{I}_2 + \nabla u_s(x_{s,i,j,k,l})) \bigr).
  \end{aligned}
\] 
Note that increasing sampling sizes $m$ lead to higher accuracy of the approximated integrals, at the cost of higher computational demands. Hence, we like to set the sampling size $m$ at a given scale $s$ preferably small. In order to catch at least the features of $I_1$ in the data term $D_{s,i,j}(u_s)$, it is reasonable to choose $m \ge s$.

For minimizing the discretized energy $E_s$ in \eqref{eq:energyDiscretized}, 
we use an iterative optimization method. 
Since the energy $E_s$ depends not only on the displacement field $u_s$ and the tensor field $w_s$, 
but also on the derivatives of $u_s$ and the finite differences of $w_s$  in $\TV^2$,
a variable splitting approach with additional variables for these expressions is necessary.  
The derivatives of $u_s\colon\Omega \to \mathbb R^2$ are polynomials 
of degree at most one on the domain $\Omega_{s,i,j}$. More precisely, from
\[
\frac{\partial}{\partial x_1} B(x) =
\begin{cases}
  1 - |x_2|, & -1 < x_1 < 0,\\
  |x_2| - 1, & \;\;\; 0 < x_1 <1,\\
\end{cases}
\quad
\frac{\partial}{\partial x_2} B(x) =
\begin{cases}
  1 - |x_1|, & -1 < x_2 < 0,\\
  |x_1| - 1, & \;\;\; 0 < x_2 <1,
\end{cases}
\]
we infer for $x \in \Omega_{s,i,j}$ the relations
\begin{equation*}
  \begin{aligned}
  \frac{\partial}{\partial x_1} u_s(x)
   & = z_{s,i,j,1} (1 - t_2) + z_{s,i,j,2} t_2, \qquad t_2 \coloneqq x_2/s - j \in (0,1),\\
  \frac{\partial}{\partial x_2} u_s(x)
   & = z_{s,i,j,3} (1 - t_1) + z_{s,i,j,4} t_1, \qquad t_1 \coloneqq x_1/s - i \in (0,1), 
  \end{aligned}  
\end{equation*}
where the finite difference coefficients $z_{s,i,j,\kappa}$, $\kappa = 1,\ldots,4$, solve the system
of equations
\begin{equation} \label{eq:z}
  \begin{aligned}
  0= h_{s,i,j,1}(u_s,z_s) &\coloneqq z_{s,i,j,1} - (u_{s,i+1,j} - u_{s,i,j})/s , \\
  0= h_{s,i,j,2}(u_s,z_s) &\coloneqq z_{s,i,j,2} - (u_{s,i+1,j+1} - u_{s,i,j+1})/s, \\
  0= h_{s,i,j,3}(u_s,z_s) &\coloneqq z_{s,i,j,3} - (u_{s,i,j+1} - u_{s,i,j})/s , \\
  0= h_{s,i,j,4}(u_s,z_s) &\coloneqq z_{s,i,j,4} - (u_{s,i+1,j+1} - u_{s,i+1,j})/s.
  \end{aligned}
\end{equation}

Note that this artificial linear constraint is crucial for applying the ADMM algorithm.
Similarly, for the differences of the function $w_s\colon\Omega \to T^2(\R^2)$ appearing in $\TV^2$, 
we introduce the finite difference variables $\omega_{s,i,j,\kappa} \in T^2(\R^2)$, $\kappa = 1,2$, which solve the system
of equations
\begin{equation}  \label{eq:w}
  \begin{aligned}
  0 = g_{s,i,j,1}(w_s, \omega_s)
	& \coloneqq 
  \begin{cases}
    \omega_{s,i,j,1} - (w_{s,i+1,j} - w_{s,i,j}), & i < a/s - 1,\\
    \omega_{s,i,j,1}, & i = a/s- 1,
  \end{cases}
  \\
  0 = g_{s,i,j,2}(w_s, \omega_s) & \coloneqq 
  \begin{cases}
    \omega_{s,i,j,2} - (w_{s,i,j+1} - w_{s,i,j}), & j < b/s - 1,\\
    \omega_{s,i,j,2}, & j = b/s  - 1.
  \end{cases}
\end{aligned}
\end{equation}
Then, we replace any occurrence of the displacement gradient $\nabla u_s\colon\Omega \to T^2(\R^2)$  
in the discretized energy $E_s$ in \eqref{eq:energyDiscretized} 
by the piecewise continuous function
\[
  z_s(x) \coloneqq \sum_{i=0}^{a/s - 1}\sum_{j=0}^{ b/s - 1}
  \begin{pmatrix}
    z_{s,i,j,1,1}(1-t_2) + z_{s,i,j,2,1}t_2& z_{s,i,j,3,1}(1-t_1) + z_{s,i,j,4,1}t_1\\
    z_{s,i,j,1,2}(1-t_2) + z_{s,i,j,2,2}t_2& z_{s,i,j,3,2}(1-t_1) + z_{s,i,j,4,2}t_1
  \end{pmatrix},
\] 
where $t_1 = x_1/s - i$ and $t_2 = x_2/s-j$.
For the function $w_s$, we replace the differences in $\TV^2_{s,i,j}$
by the variable $\omega_{s} \coloneqq (\omega_{s,i,j,\kappa})_{i,j,\kappa}$. 
Consequently, if the constraints \eqref{eq:z} and \eqref{eq:w} are satisfied, we can rewrite the summands appearing in \eqref{eq:energyDiscretized} as
\[
  \begin{aligned}
  &D_{s,i,j}(u_s, z_s) 
	 \coloneqq 
	\frac{1}{m^{2}} \sum_{k,l=0}^{m-1} 
  \dx_{\mathbb S^3 / \mathcal S} \left( I_1(x_{s,i,j,k,l}), 
	\overline{q(\theta_{z_s}(x_{s,i,j,k,l}))} \odot I_2\bigl(x_{s,i,j,k,l} + u_s(x_{s,i,j,k,l})\bigr) \right), \\
  &\TV^1_{s,i,j}(w_s, z_s) 
	 \coloneqq 
	\frac{1}{m^{2}} \sum_{k,l=0}^{m-1} \| z_s(x_{s,i,j,k,l}) - w_{s,i,j}\|_F,\\
  &\TV^2_{s,i,j}(\omega_s) 
	 \coloneqq 
	\frac{1}{s} \big( \| \omega_{s,i,j,1} \|_F^2 + \| \omega_{s,i,j,2} \|_F^2 \big)^\frac12,\\
  &F_{s,i,j}(z_s) 
	 \coloneqq  
	\frac{1}{m^{2}} \sum_{k,l=0}^{m-1} f\bigl(\det(\mathbf{I}_2 + z_s(x_{s,i,j,k,l})) \bigr),
  \end{aligned}
\]
where $\theta_{z_s}(x_{s,i,j,k,l})$ denotes the angle defined in \eqref{eq:polarAngle2d} for the matrix $\mathbf{I}_2 + z_s (x_{s,i,j,k,l})$.
In summary, we get an extended form of \eqref{eq:energyDiscretized}, which we denote with $E_s(u_s, w_s, z_s, \omega_s)$.

\begin{remark}  \label{rem:discrete_modifications}
  The discrete version for the modification of the model \eqref{eq:var_problem_orig} introduced in Subsection~\ref{sec:well-posed_modifications_2} can be treated in a similar way. 
	For instance, the discrete version of the $\TV^2$-model is obtained by setting $\alpha_1 = 0$ in  \eqref{eq:energyDiscretized} and adding the constraints $z_s(i+\frac12,j+\frac12) = w_{s,i,j}$, $i = 0,\dots, a/s$, $j=0, \dots, b/s$, to the optimization problem \eqref{eq:energyDiscretized}.
	Then, similar algorithms as proposed in the next section can be derived.
\end{remark}

\section{Optimization Algorithm} \label{sec:alg}
%
In this section, we describe the optimization algorithm for the non-smooth, non-convex
and high dimensional problem \eqref{eq:energyDiscretized}.
A reasonable and efficient method for solving constrained optimization problems 
is the Augmented Lagrangian Method (ALM), also known as Method of Multipliers, see \cite{He69,Po72,Ro76a},
which enables the use of unconstrained optimization solvers. 
Global convergence results under relatively mild conditions, 
even for non-smooth and non-convex optimization problems, were proved, e.g., in \cite{AnBiMaSch08,BiFlMa10}. 
We use a particularly efficient variant, 
the Alternating Direction Method of Multipliers (ADMM), 
which in the context of convex optimization provides global convergence.
It goes back to \cite{GM76,GM75} and for an overview we refer to \cite{BPCPE11,glowinski14}.
In general, the ADMM cannot be applied reliably 
to non-convex and non-smooth problems. 
Recently, some promising results for particular problems were given in \cite{WaYiZe19}.
In the following, we briefly show how the ADMM can be applied to our problem.
Indeed all ADMM steps can be incorporated within a multilevel approach and can be computed in an efficient way, where we observe numerical convergence. We are not aware of any other approach for the registration of EBSD data in the literature.

\subsection{Algorithm}
We start by noting that the augmented Lagrangian function for the minimization of $E_s(u_s, w_s, z_s, \omega_s)$ under the constraints \eqref{eq:z} and \eqref{eq:w} is given by
\begin{equation}
  \label{eq:energyAugmentedLagrangian}
    L_\mu (u_s, w_s, z_s, \omega_s, \lambda_h, \lambda_g) 
		\coloneqq 
		E_s(u_s, w_s, z_s, \omega_s) + \tfrac{\mu}{2} H_{s}(u_s, w_s, z_s, \omega_s,\lambda_h,\lambda_g), \qquad \mu >0,
\end{equation}  
with
  \begin{align}
   H_{s}(u_s, w_s, z_s, \omega_s, \lambda_h, \lambda_g)   
  & \coloneqq 
	\sum_{i=0}^{a/s -1}\sum_{j=0}^{ b/s -1} \sum_{\kappa =1}^4  
	\| h_{s,i,j,\kappa}(u_s,z_{s})  	+ \tfrac{1}{\mu} \lambda_{h,i,j,\kappa} \|_2^2   \\
  & \; + \sum_{i=0}^{ a/s -1}\sum_{j=0}^{ b/s -1} \sum_{\kappa=1}^2  
	\| g_{s,i,j,\kappa}(w_s,\omega_{s}) +  \tfrac{1}{\mu} \lambda_{g,i,j,\kappa}\|_F^2, 
  \end{align}
primal variables 
\[
  \begin{aligned}
  u_{s,i,j} &\in \mathbb R^2 , & i&=0,\dots, a/s, & j&=0,\dots, b/s, \\
   w_{s,i,j} &\in \mathbb R^{2, 2} , & i&=0,\dots, a/s - 1, & j&=0,\dots,  b/s - 1,\\
   z_{s,i,j} & \coloneqq (z_{s,i,j,\kappa})_{\kappa=1}^4 \in \mathbb R^{2,4}, & i&=0,\dots, a/s - 1, &  j&=0,\dots,b/s  - 1, \\
   \omega_{s,i,j} & \coloneqq (\omega_{s,i,j,\kappa})_{\kappa=1}^2 \in \mathbb R^{2,2,2}, & i&=0,\dots, a/s - 1, & j&=0,\dots,  b/s - 1,\\
  \end{aligned}
\]
and dual variables 
\[
  \lambda_{h} \coloneqq (\lambda_{h,i,j,\kappa})_{\kappa=1}^4 \in \mathbb R^{2,4}, \;
	\lambda_{g} \coloneqq (\lambda_{g,i,j,\kappa})_{\kappa=1}^2 \in \mathbb R^{2,2,2},
	\quad i=0,\dots, a/s-1, \,\, j=0,\dots, b/s-1.
\]
Then, the ADMM aims to solve the constrained problem starting with an initial guess
$
(u_s^0, w_s^0, z_s^0, \omega_s^0, \lambda_h^0,\lambda_g^0)
$
iteratively based on the alternating primal-dual procedure
\begin{align}
  (u_s^{r+1}, w_s^{r+1})      
	& \coloneqq \argmin_{(u_s, w_s)} L_\mu (u_s, w_s, z_s^{r}, \omega_s^{k}, \lambda_h^{r}, \lambda_g^{r}), 
	\label{eq:UVmin}\\
  (z_s^{r+1}, \omega_s^{r+1}) 
	& \coloneqq \argmin_{(z_s, \omega_s)} L_\mu(u_s^{r+1}, w_s^{r+1}, z_s, \omega_s, \lambda_h^{r}, \lambda_g^{r}), 
	\label{eq:ZWmin}\\
    (\lambda_h^{r+1}, \lambda_g^{r+1})
    & \coloneqq 
	(\lambda_h^{r}, \lambda_g^{r}) + \mu \,( h_s^{r+1}, g_s^{r+1}),
	\label{eq:lambdaAscent} 
\end{align}
see \cite{BSS2016}.
Unfortunately, we cannot give an explicit solution for the primal problems \eqref{eq:UVmin} and \eqref{eq:ZWmin}. 
Instead, we minimize the augmented Lagrangian $L_\mu$ 
for the primal variables iteratively using the algorithms in Appendix \ref{app:B}.
More precisely, the optimal $z_s$, $u_s$ and $w_s$ are computed by steepest descent methods 
with inexact line search, c.f.\ Algorithm~\ref{alg:LineSearch}, 
whereas the optimal $\omega_s$ is computed analytically. 
Unfortunately, descent methods converge in general only towards local minimizers. 
Furthermore, we like to emphasize that the function $L_\mu$ 
might not be differentiable at particular points.
At such points we use a subgradient instead of the gradient. 
Since these points of non-differentiability occur only where the minimum of the distance or norms is achieved, the proposed algorithms might have problems only close to local minimizers, where non-differentiability is present. 
In such cases, it is  difficult to determine the correct step size by the line search.
However, our numerical experiments indicate that the proposed algorithms do perform well even in such corner cases. 
\\[2ex]
\textbf{$(u_s, w_s)$ - Minimization.}
The minimization in \eqref{eq:UVmin} w.r.t.\ the variables $u_s$, $w_s$ can be done separately.
For $u_{s}$ we aim to minimize the sums
\[
  S^1_{s,i,j}(u_s) \coloneqq D_{s,i,j}(u_s, z_s) + 
	\tfrac{\mu}{2} \sum_{\kappa=1}^{4} \bigl\|h_{s,i,j,\kappa}(u_s, z_{s}) +  \tfrac{1}{\mu} \lambda_{h,i,j,\kappa}\bigr\|_2^2 ,
\]
independently and parallel for $i=0, \dots, a / s - 1$, $j=0,\dots, b/s  - 1$. 
However, since $S^{1}_{s,i,j}$ depends on the variables $u_{s,i+k,j+l}$, $k,l \in \{0,1\}$, 
we propose to decouple the optimization as described in Algorithm~\ref{alg:Umin}. 
The advantage of the decoupling is that different step lengths 
can be taken for different regions of the displacement field $u_s$.

Similarly, for $w_s$ we aim to minimize the sums
\[
  S^2_{s,i,j}(w_s) 
	\coloneqq 
	\alpha_1 \TV^1_{s,i,j}(z_s, w_s) + \tfrac{\mu}{2}  \sum_{\kappa=1}^2 
	\bigl\| g_{s,i,j,k}(w_s,\omega_{s}) + \tfrac{1}{\mu} \lambda_{g,i,j,\kappa} \bigr\|_F^2 
\]
independently and parallel for $i=0, \dots, a / s - 1$, $j=0,\dots, b/s - 1$ 
as described in Algorithm~\ref{alg:Vmin}. 
\\[2ex]
\textbf{$(z_s, \omega_s)$ - Minimization.}
The minimization in \eqref{eq:ZWmin} of the variables $z_s$ and $\omega_s$ can be done separately.
For $z_s$ the sums 
\[
  \begin{aligned}
  S^3_{s,i,j}(z_s) & \coloneqq D_{s,i,j}(u_s, z_s) + \alpha_1  \TV^1_{s,i,j}(z_s, w_s) + \beta F_{s,i,j}(z_s) \\ 
  & + \tfrac{\mu}{2} \sum_{\kappa=1}^{4}  \bigl\|h_{s,i,j,\kappa}(u_s, z_{s}) +  \tfrac{1}{\mu}
	\lambda_{h,i,j,\kappa}\bigr\|_2^2 
  \end{aligned}
\]
can be minimized independently and parallel for $i=0, \dots, a / s - 1$, $j=0,\dots,  b/s - 1$. 
The main difficulty in the optimization of $S^3_{s,i,,j}$ arises from the non-differentiability of $D_{s, i, j}$ along higher dimensional varieties.
This is caused by the level sets of the function in \eqref{eq:polarAngle2d}, which reads for our setting as
\begin{equation}
  \label{eq:angle_z}
   \theta(z) \coloneqq \arctan \Bigl( \frac{z_{2,1} - z_{1,2}}{ z_{1,1} + z_{2,2} + 2} \Bigr), \qquad z \in T^2(\R^2),
\end{equation}
and the non-differentiability of the distance $\dx_{\S^3 /\mathcal S}(\cdot , q)$, $q \in \S^3/\mathcal S$ at $q$. 

In order to obtain more suitable descent directions, we decompose the subgradient 
\[
  \nabla_{z_{s,i,j}} L_\mu =  \nabla_{z_{s,i,j}} S^3_{s,i,j}
\]
into a gradient ``parallel'' and a gradient ``orthogonal'' to the non-differentiable variety. 
The ``orthogonal'' gradient is defined by
\begin{equation}   \label{eq:zsGradOrtho}
  \nabla_{z_{s,i,j}}^{\perp} L_\mu \coloneqq \nabla_{z_{s,i,j}} D_{s,i,j},
\end{equation}
since a change in the data term $D_{s, i,j}$ drives directly 
a change in the level sets of the rotation angles, see \eqref{eq:angle_z}. 
Then, the ``parallel'' gradient is given by the orthogonal projection 
\begin{equation}  \label{eq:zsGradPar}
  \nabla_{z_{s,i,j}}^{\parallel} L_\mu 
	\coloneqq 
	\nabla_{z_{s,i,j}} L_\mu - 
  \frac{\bigl\langle \nabla_{z_{s,i,j}}^{\perp} L_\mu, \nabla_{z_{s,i,j}} L_\mu \bigr\rangle}{\bigl\langle 
	\nabla_{z_{s,i,j}}^{\perp} L_\mu, \nabla_{z_{s,i,j}}^{\perp} L_\mu \bigr\rangle}   \nabla_{z_{s,i,j}}^{\perp} L_\mu,
\end{equation}
which is likely to point in directions parallel to the level sets of the rotation angles.
Using alternately the descent directions $-\nabla_{z_{s,i,j}}^{\perp} L_\mu$ and $-\nabla_{z_{s,i,j}}^{\parallel}$,
we arrive at Algorithm~\ref{alg:Zmin}.

For the variable $\omega_s$, we minimize the sums
\[
  S^4_{s,i,j}(\omega_s) 
	\coloneqq 
	\alpha_2 \TV^2_{s,i,j}(\omega_s) + 
	\tfrac{\mu}{2} \sum_{\kappa=1}^2 
	\bigl\| g_{s,i,j,k}(w_s,\omega_{s}) + \tfrac{1}{\mu} \lambda_{s,i,j,\kappa}\bigr\|_F^2 , 
	\]
analytically and parallel for $i=0, \dots,  a / s - 1$, $j=0,\dots,  b/s - 1$. 
Here, we utilize that the solution of
\[
  \argmin_{x \in \mathbb R^d} \tfrac{\mu}{2} \|x - y\|_2^2 
	+ \beta \| x \|_2, \qquad y \in \mathbb R^d, \qquad \mu, \beta > 0
\] 
is given by the grouped soft shrinkage 
\[ 
  x^{*} =  
  \begin{cases}
      y  \bigl(1 - \tfrac{\beta}{\|\mu y \|_2}\bigr)   , & \|\mu y  \|_2 > \beta,\\
    0, & \text{else}.
  \end{cases}
\]

\subsection{Multilevel Approach and Implementation}
Minimizing $E_s(u_s, w_s, z_s, \omega_s)$ under the constraints \eqref{eq:z} and \eqref{eq:w} is a highly non-convex task.
This results in many local minima, in particular for high image resolutions with many degrees of freedom. 
In order to find good solutions, we  apply a multilevel approach, 
where we successively increase the approximation resolution. More precisely, 
for a given scaling factor $s_l > 0$ on level $l \in \N_0$, 
we apply the ADMM \eqref{eq:UVmin}-\eqref{eq:lambdaAscent} and use the computed displacement $u_{s_l}$ as initial guess for the next level by setting $s_{l+1} \coloneqq s_l / 2$. 
On the coarsest level $l=0$, we take the largest scale $s_0 := \gcd(a,b)$. The finest level $l_{\max}$ is achieved if $s_{l_{\max}} < 1$. Hence, domains $\Omega = (0,a) \times (0, b)$ with sides $a$ and $b$ having large common divisors, are preferred for our multilevel approach. Furthermore, we choose a relatively high sampling density by setting the sampling factor $m_0 := 2 s_0$, which is decreased at every subsequent level by setting $m_{l+1} \coloneqq \max\{m_l / 2, 2\}$. The use of high sampling factors $m_l$ allows us to avoid additional filtering steps, such as smoothing of the images $I_1$, $I_2$, as it is usually done for other registration or optical flow approaches, see, e.g., \cite{SRB13}. 
For the $(u_s, w_s)$- and $(z_s, \omega_s)$-minimization steps \eqref{eq:UVmin} and \eqref{eq:ZWmin},
it is sufficient to fix the maximal number of iterations $k_{\max} \coloneqq 5$ in Algorithms~\ref{alg:Umin}-\ref{alg:Zmin}. 
At each level $l$, we use $l \cdot 1000$ ADMM-iterations.
In cases where large deformations are expected, it might by advantageous to start with a small penalty parameter $\mu_0$ at level $l=0$ and increase it at each level. 
This adds more flexibility and efficiency for the optimization on coarser levels.

Finally, we like to remark that our discretization and optimization strategy is well suited for the use of parallel computing devices. 
Hence, we implemented the proposed algorithms for GPU devices using Python together with the CUDA toolkit.
Here, we use the Python interface provided by PyCUDA \cite{kloeckner_pycuda_2012} to implement the parallel Algorithms~\ref{alg:Umin}-\ref{alg:Zmin} with the CUDA programming language \cite{CUDA08}.
The software will be publicly available.

\section{Numerical Experiments} \label{sec:numerics}
In this section, we demonstrate the performance of our registration models 
and optimization algorithms 
on synthetic data as well as real world data. 
In Example~\ref{sec:tearing_square}, we show that the TGV-model \eqref{eq:var_problem1} 
and the ``higher regularity'' $\TV^2$-model \eqref{eq:var_problem2} 
are able to reconstruct a displacement field with a jump discontinuity. 
In Example~\ref{sec:ebsd_measurements}, we recover the rigid rotations 
between two rotated real-world EBSD datas sets.
Here, our model is in perfect accordance 
with the physical conditions.
Further, we discuss the influence of the regularization parameters in our model. 
Finally, in Example~\ref{sec:simulated_deformation}, we take data from a simulation where ice crystals are deformed under shear stress. Using the $\TV^2$-model, we are able to reconstruct the deformation.

The regularization parameters in our models are chosen to provide visual appealing reconstruction results for a variety of scenarios. More precisely, if not stated otherwise we use $\alpha_1 = 0.1$, $\alpha_2 = 0.5$ (TGV), $\alpha = 0.5$ ($\TV^2$), $\beta = 0.1$ (determinant). For ADMM we usually take $\mu = 1$, where smaller parameters appear leading to larger violations of the equality constraints and larger parameters $\mu$ slow down the overall convergence. The variables  $u_s, w_s, z_s, \omega_s, \lambda_h, \lambda_g$ of the multilevel method are initialized by zero. 

The orientations $q=(s,v) \in \S^3$ of the EBSD data in Figure~\ref{fig:tearing_square} -- \ref{fig:iceImages} are colorized by taking the absolute values of the vectorial part $v$ as color coordinates in RGB space. For more sophisticated color codings we refer to \cite{NoHi17}.
The symmetry group $\mathcal S$ is given via the phase of the EBSD measurements and specified below for each real-world example.

\subsection{Tearing Square} \label{sec:tearing_square}
In this example, we  demonstrate that both the TGV-model and the ``higher regularity'' $\TV^2$-model 
are able to recover the displacement field with a jump discontinuity. Recall, that our discrete model 
is based on continuous basis functions, so that we can recover jump discontinuities only in the 
limit of the refinement process. Nevertheless, we shall see that the TGV-model \eqref{eq:var_problem1} is able to resolve the jump discontinuities almost perfectly.

The images $I_1$ and $I_2$ of size $256 \times 256$, for which we reconstruct the displacement field, are given in the first row of Figure~\ref{fig:tearing_square}.
The black background in both images is associated to the orientation $(1, 0, 0, 0) \in \S^3$.
Image $I_1$ contains a square region of size 128x128 of constant orientation  $(0, 1 / \sqrt{2}, 1/ \sqrt{2}, 0) \in \S^3$.
In image $I_2$, the left half of the square is rotated to the left and the right half is rotated to the right, each by an angle of 30 degrees. 
Hence, the orientations are  $(0, 1/2, \sqrt{3}/2, 0) \in \S^3$ and $(0, \sqrt{3}/2, 1/2, 0) \in \S^3$, respectively.
For this example, we assume no further group symmetry, i.e., $\mathcal S = \{ \pm 1 \}$.

In Figure~\ref{fig:tearing_square}, we show the reconstructed displacement field of both methods. 
We observe that the jump part is better resolved by the TGV-model, 
which is easily explained by the higher order terms in the $\TV^2$-model. 
However, it is interesting that also the $\TV^2$-model is able to handle such large jump discontinuities. 

\begin{figure}
    \begin{center}
    \includegraphics[width=0.49\textwidth]{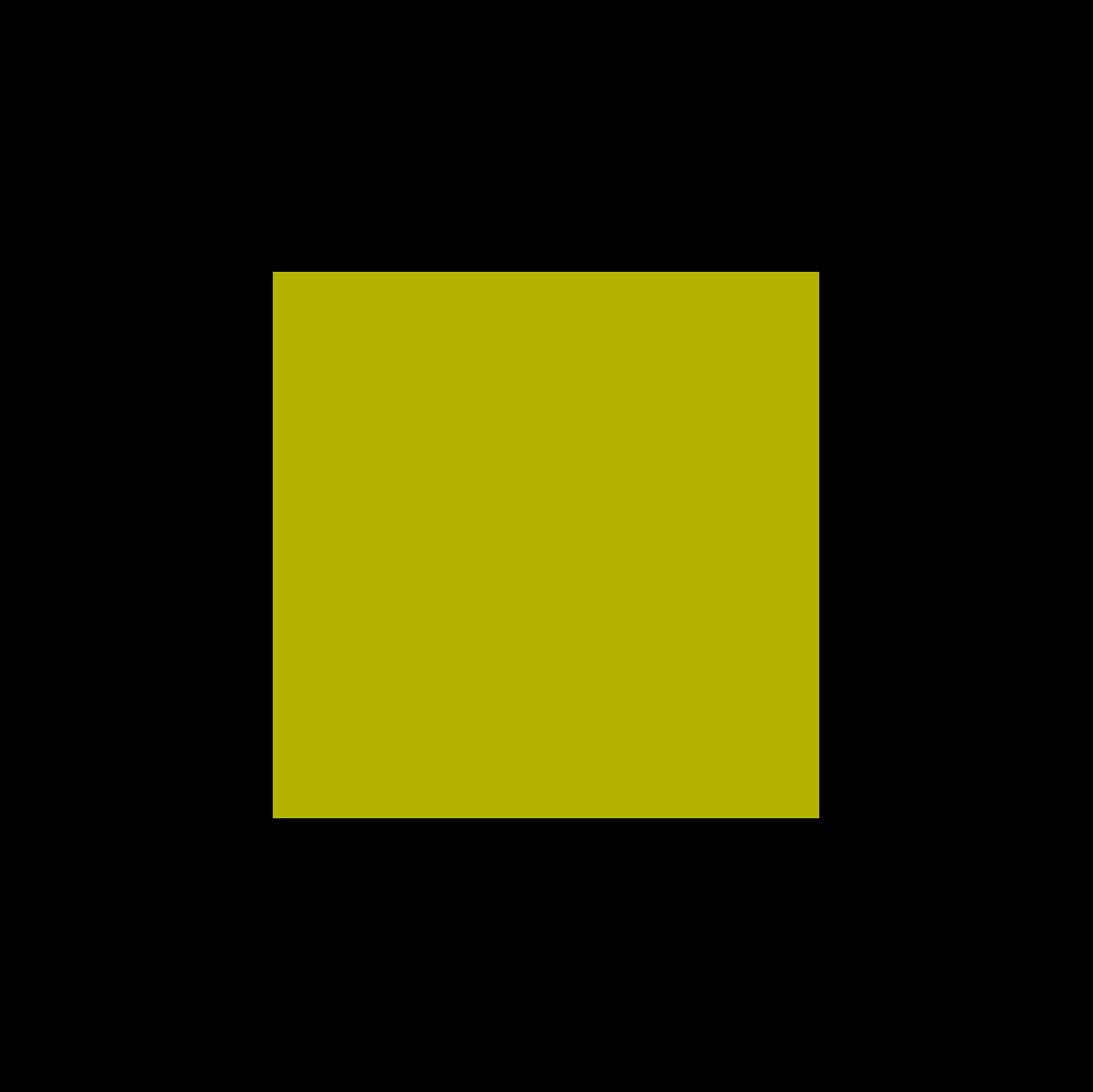}
    \includegraphics[width=0.49\textwidth]{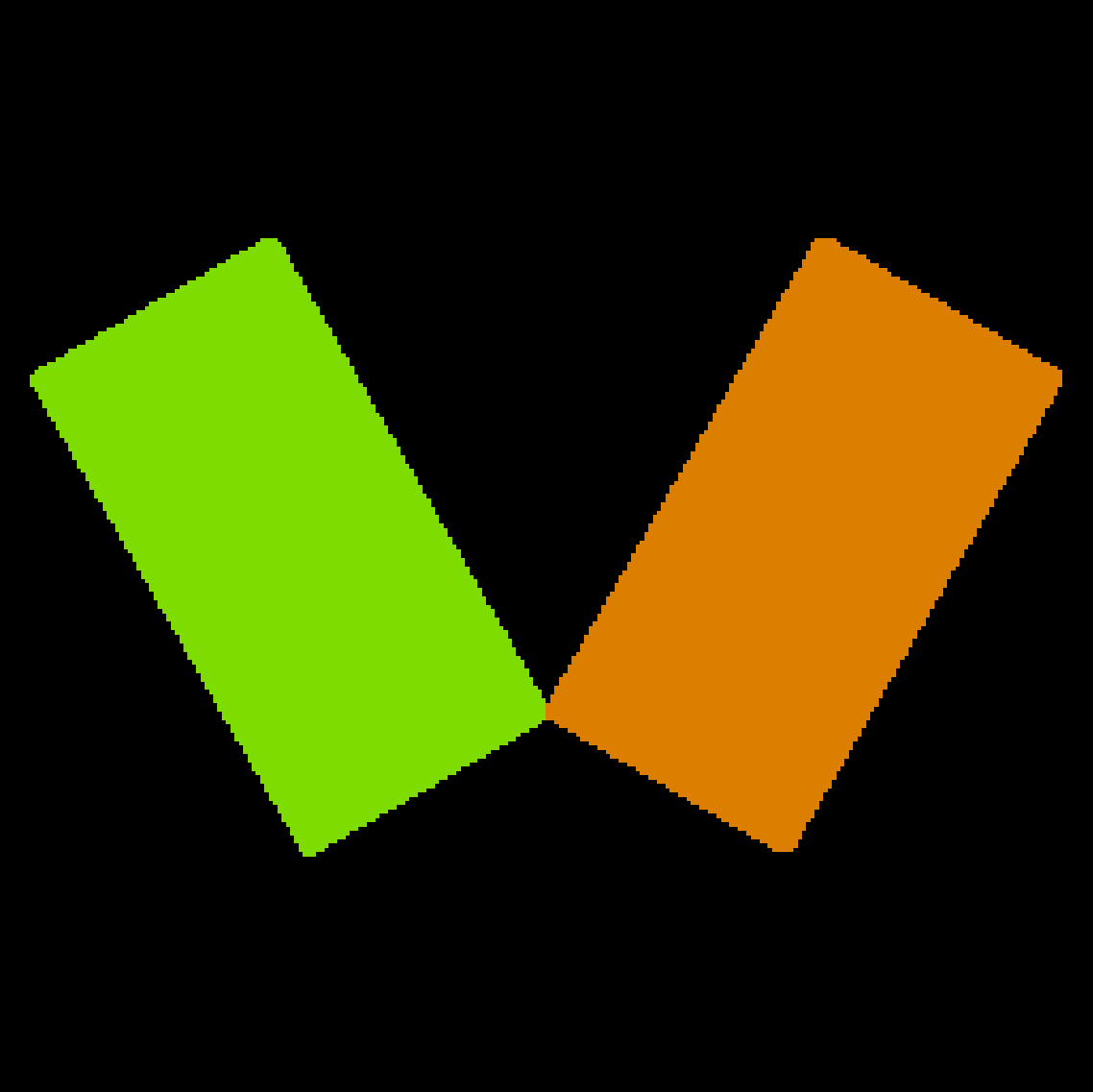}
    \includegraphics[width=0.30\textwidth]{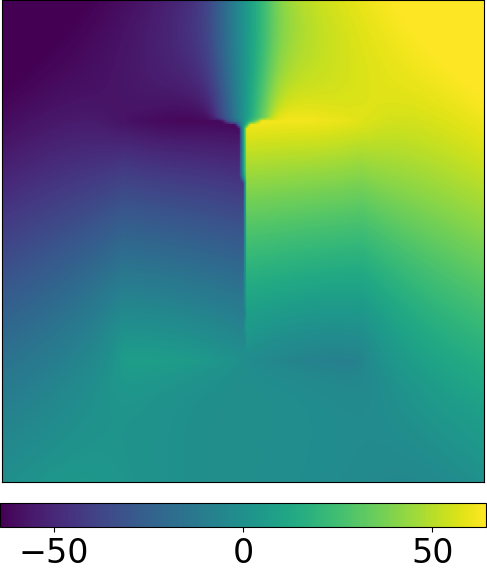}
    \includegraphics[width=0.314\textwidth]{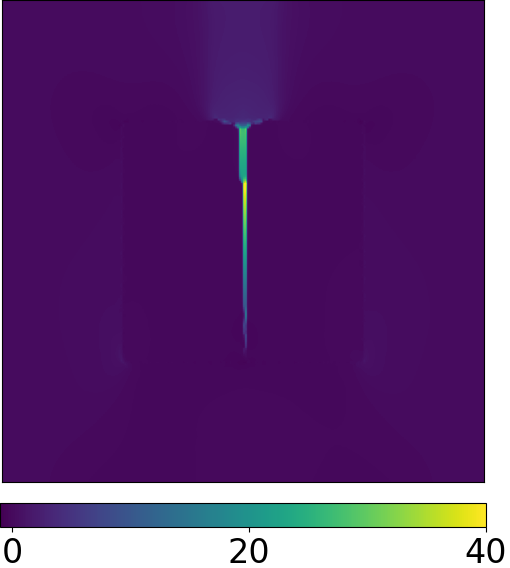}
    \includegraphics[width=0.36\textwidth]{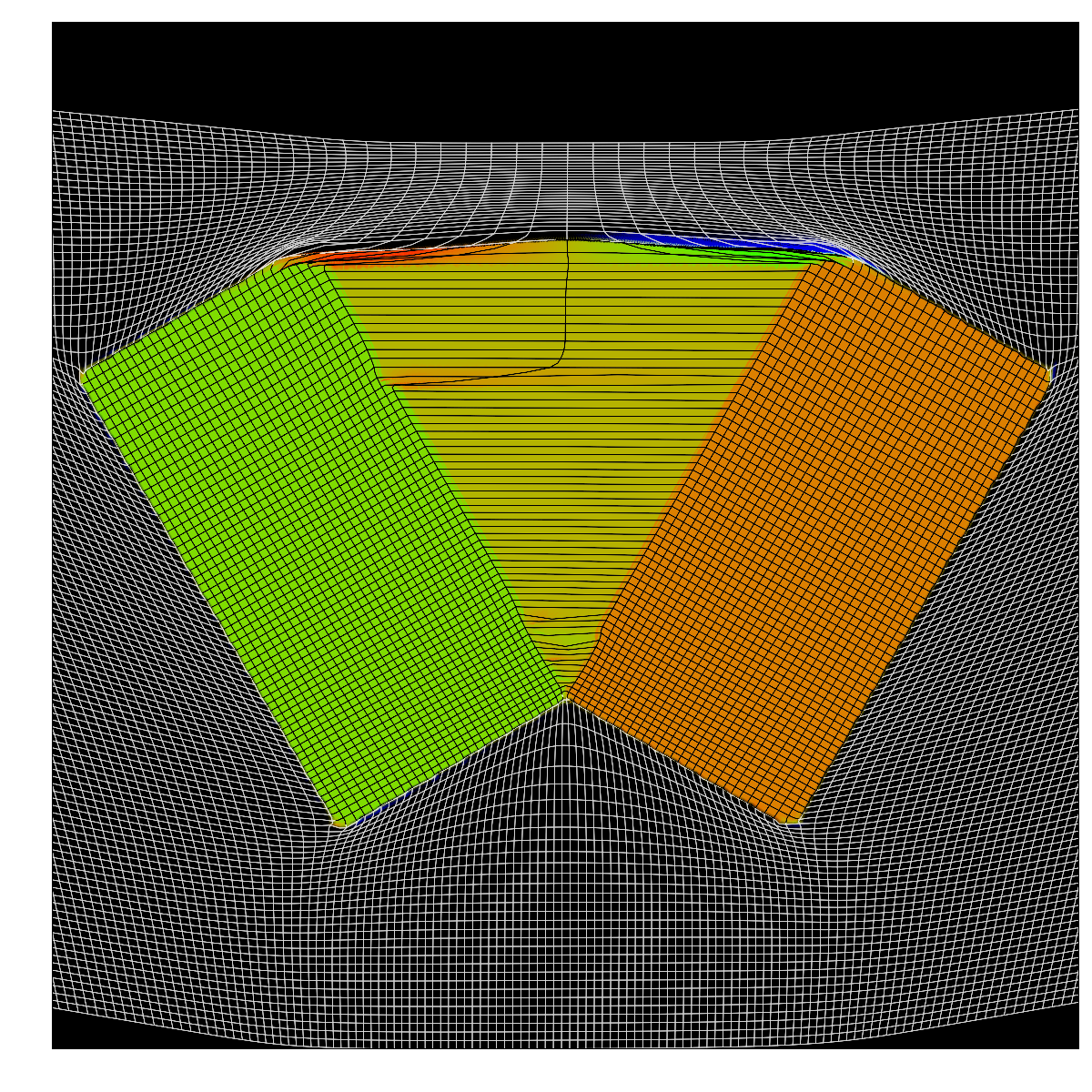}
    \includegraphics[width=0.30\textwidth]{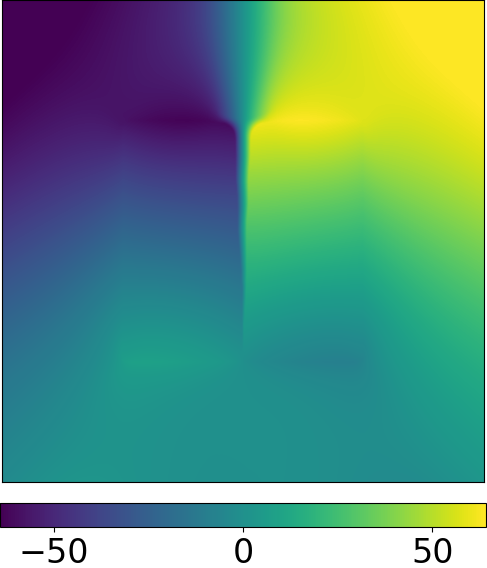}
    \includegraphics[width=0.314\textwidth]{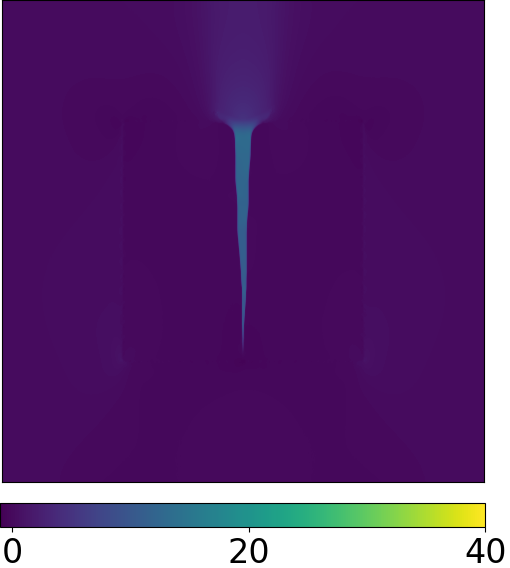}
    \includegraphics[width=0.36\textwidth]{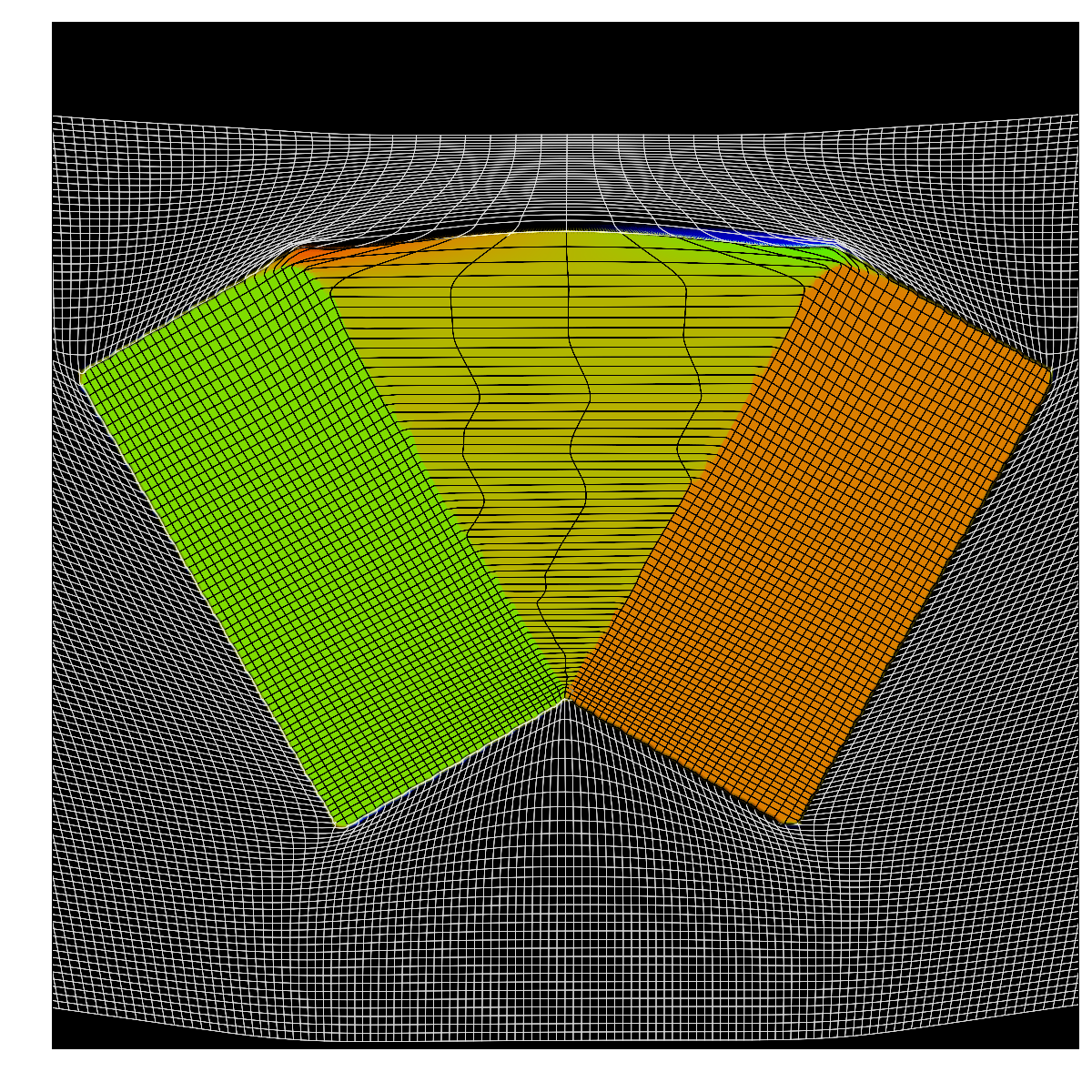}
    \end{center}
    \caption{ 
    Artificial EBSD data of a tearing square. Top row: Image $I_1$ (left) representing the original square and Image $I_2$ (right) representing the teared square. Both images are of size $256\times 256$ and visualized in RGB space. Middle row: Reconstruction results of the TGV-model. Bottom row: Reconstruction results of the $\TV^2$-model. From left to right: a) The first component of the displacement field $u_1$ showing the vertical jump discontinuity. b) Image of the derivative $\frac{\partial}{\partial x_1} u_1$. c) Visualization of $\mathrm R(\nabla \varphi) I_1 \circ \varphi^{-1}$ overlaid with a grid visualizing the reconstruction of $\varphi$. }
  \label{fig:tearing_square}
\end{figure}

\subsection{Reconstruction of Rotated EBSD Measurements}\label{sec:ebsd_measurements}

Next, we deal with data from real EBSD measurements of two different samples.
The first sample is fully measured, whereas the second one has corrupted data.
Both samples are measured in two positions, which differ by a rotation with axis almost perpendicular to the surface plane.
We use the first sample to demonstrate the differences 
between our model, which incorporates $R(\nabla \varphi)$, 
and the naive approach, where the orientation of the $\SO(3)/\mathcal S$-data is not changed by the transformation.
The second example illustrates the influence of the regularization parameters on the reconstruction. 
For both examples we use the $\TV^2$-model.

The EBSD data of the first sample are measured on a zirconium-hydrogen alloy zircaloy-4, which is typically used for constructive components in the nuclear power industry due to combination of excellent corrosion resistance, good neutron penetration and suitable mechanical properties.
This sample only consists of the hexagonal phase, which has the symmetry group
'6/mmm' \cite{Fuloria.2016,Pshenichnikov.2015}.
The position of image $I_1$ and image $I_2$ differs by a rotation of 50 degrees, as can be seen in the top row of Figure~\ref{fig:rotatedSample_ralf_50_degree}. 
The bottom row of Figure~\ref{fig:rotatedSample_ralf_50_degree} depicts the error of the 
reconstructed transformation $\varphi$, where we measure the distance of image $I_1$ to the rotation aware transformed image $R(\nabla \varphi) ^\tT I_2 \circ \varphi$ (left) and to the naively transformed image $I_2 \circ \varphi$ (right), respectively. We clearly observe that the rotation influences the orientations in the EBSD measurements. 
In the rotation aware model, we observe that the orientations of the matched grains are almost perfectly aligned (bottom, left). If we do not take the rotational change into account, there appear large differences in the corresponding orientations. 
\begin{figure}
  \begin{center}
  \includegraphics[width=0.442\textwidth]{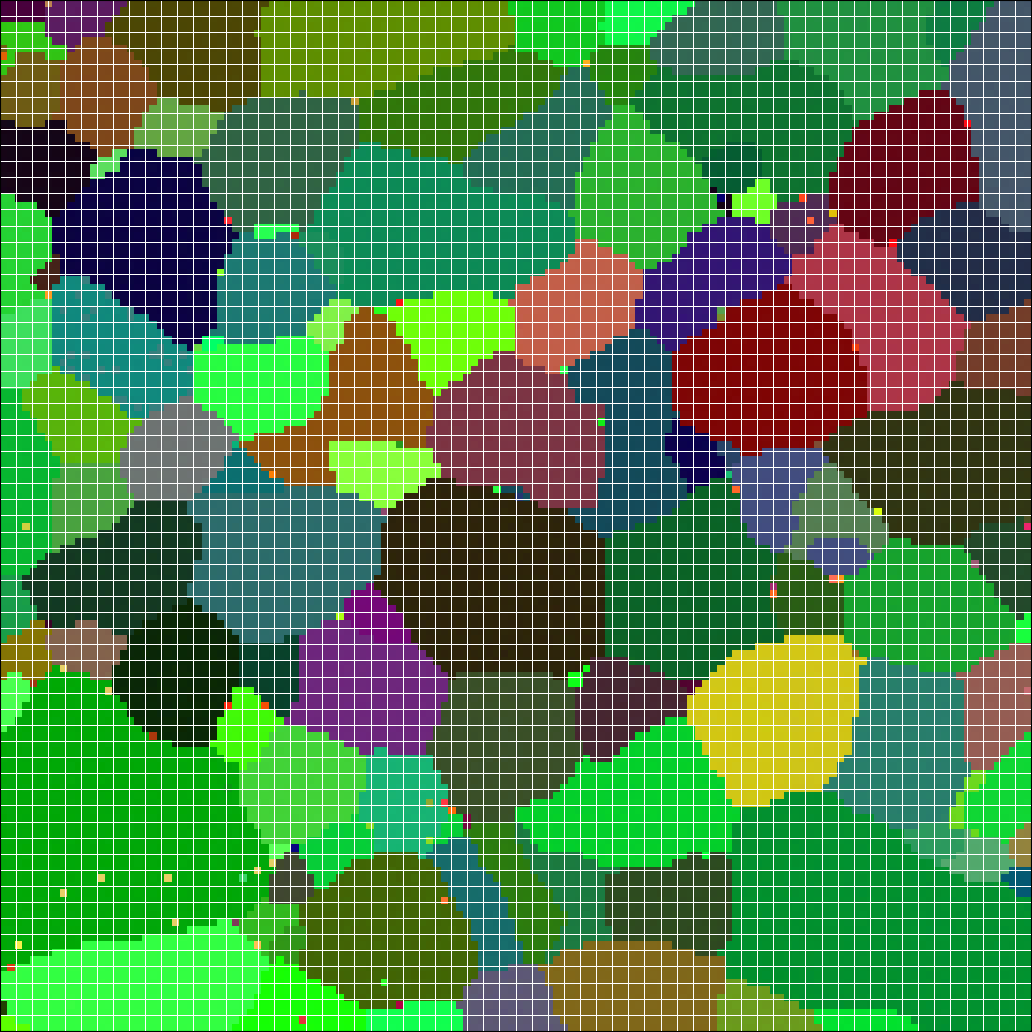}
  \includegraphics[width=0.535\textwidth]{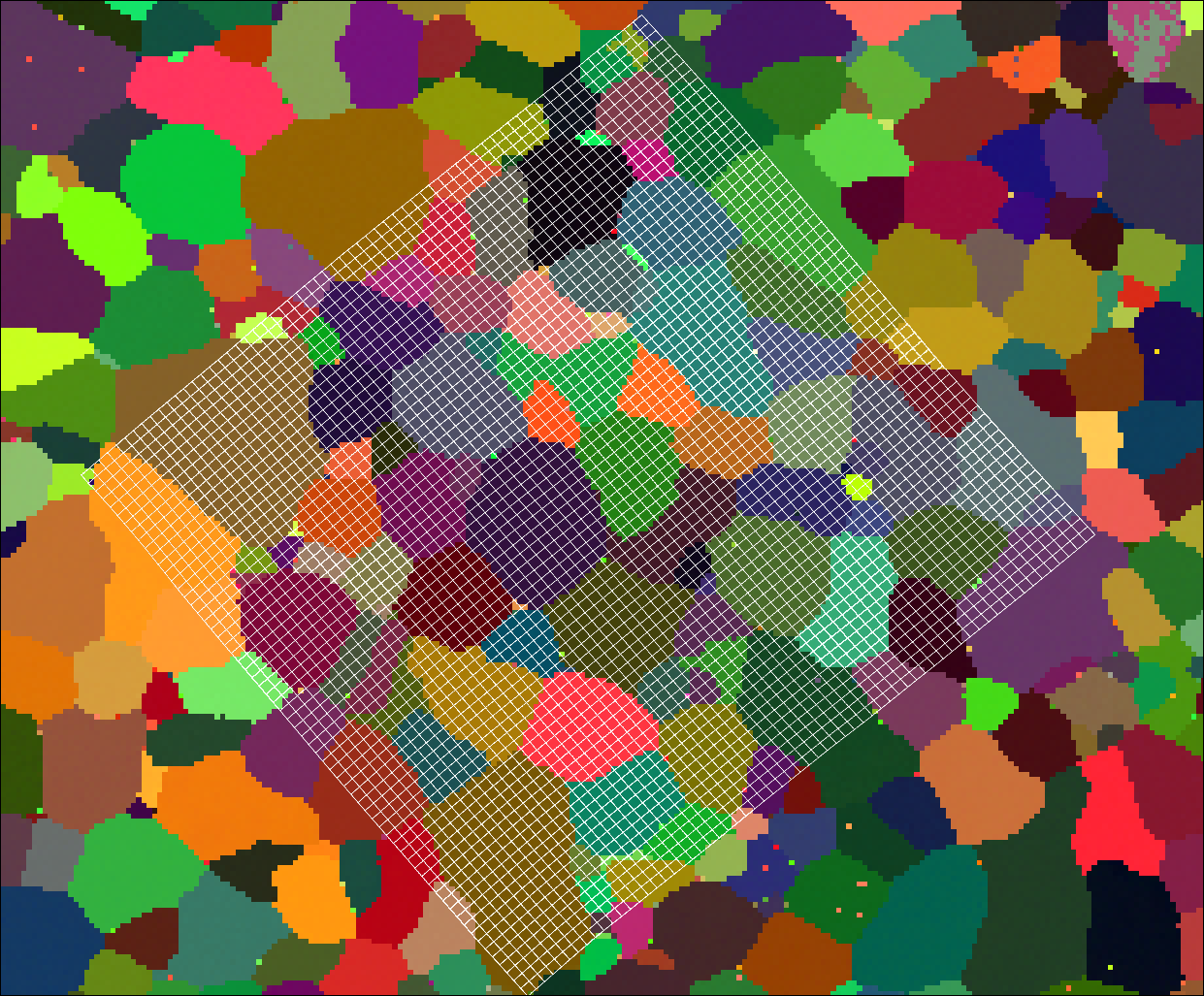}
  \includegraphics[width=0.485\textwidth]{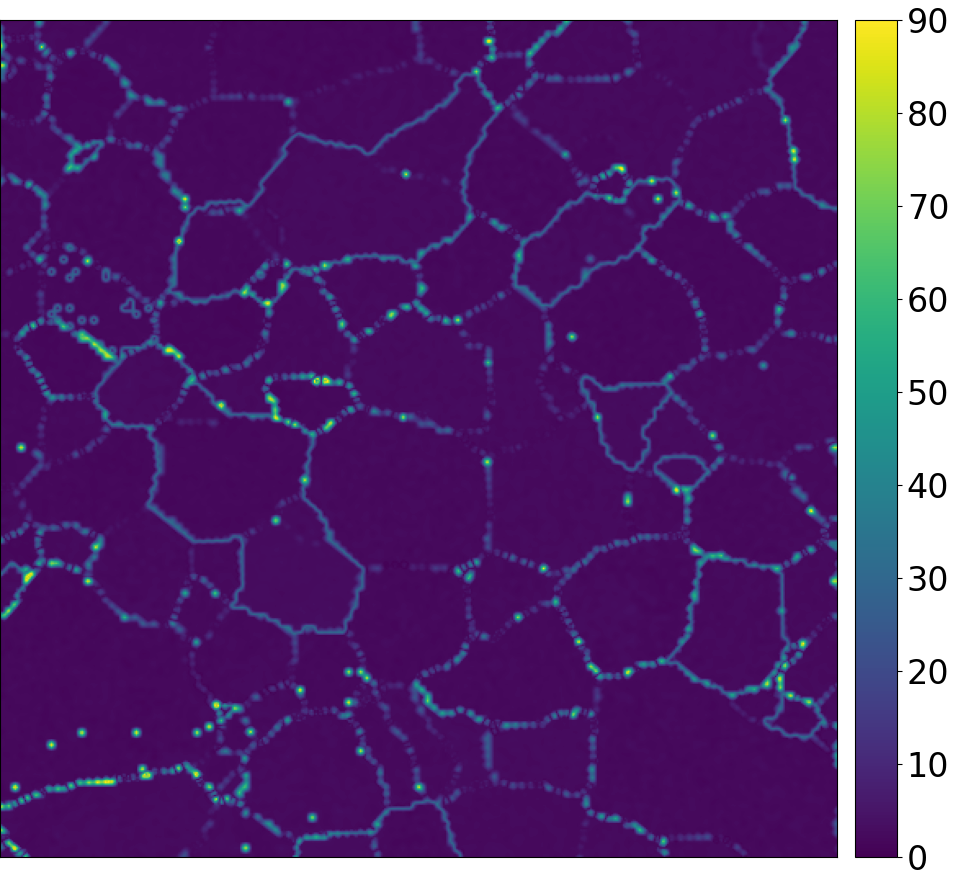}
  \includegraphics[width=0.485\textwidth]{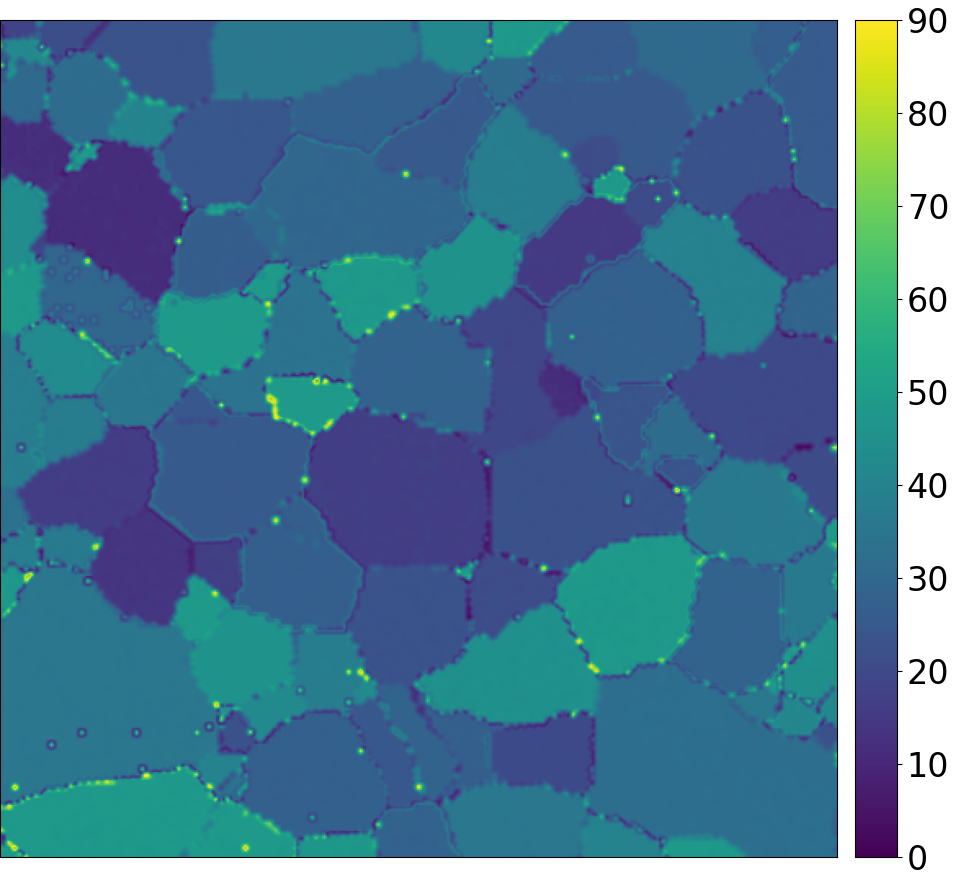}
  \end{center}
  \caption{ 
  Real EBSD data of a zirconium-hydrogen alloy zircaloy-4 measured in two different positions. Top row: Image $I_1$ (left) and image $I_2$ (right), both overlaid with a grid visualizing the reconstructed transformation (white grid). 
	Bottom row: Reconstruction error in degrees for our model (left) 
	and the model without orientation incorporation (right).
	The error for the naive model is much higher.}
  \label{fig:rotatedSample_ralf_50_degree}
\end{figure}

The second sample is a metastable austenitic so called TRIP (transformation induced plasticity) steel.
Deformation may induce both a transformation of the metastable austenitic phase (symmetry group '432') to martensitic phase (symmetry group '4/mmm') or a formation of crystallographic twins that lead to a change of macroscopic properties \cite{Schmidt.2018}.
In this case, a deformation was induced by an indentation in scope of a hardness measurement.
As EBSD is very surface sensitive, the indenter imprints can not be measured and appear as quasi diamond-shaped areas \cite{Brodusch.2018}.

For our experiment, we choose the phase to which the  majority of the grains corresponds to, namely '432'. The remaining phases and the three indenter imprints are filled 
by the constant orientation $(1,0, 0, 0) \in \S^3$ (black color) \emph{in both images} $I_1$ and $I_2$.
In the left column of Figure~\ref{fig:rotatedSample}, we can see that the position between the images $I_1$ and $I_2$ differs by a rotation of about 8 degrees. 
Moreover, we observe for decreasing regularization parameter $\alpha$ stronger deformations in areas of filled data. This behavior is expected for our model by to the influence of the rotation field $R(\nabla \varphi)$ in the data term, see right column in Figure~\ref{fig:rotatedSample}. 

\begin{figure}
    \begin{center}
    \includegraphics[width=0.45\textwidth]{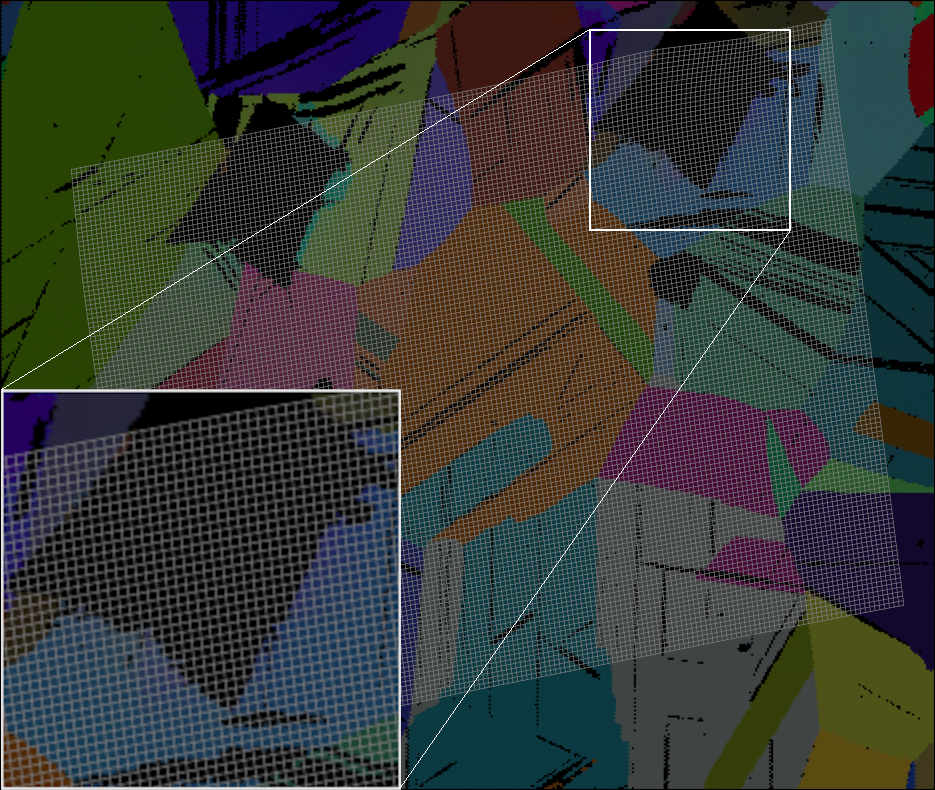}
    \includegraphics[width=0.475\textwidth]{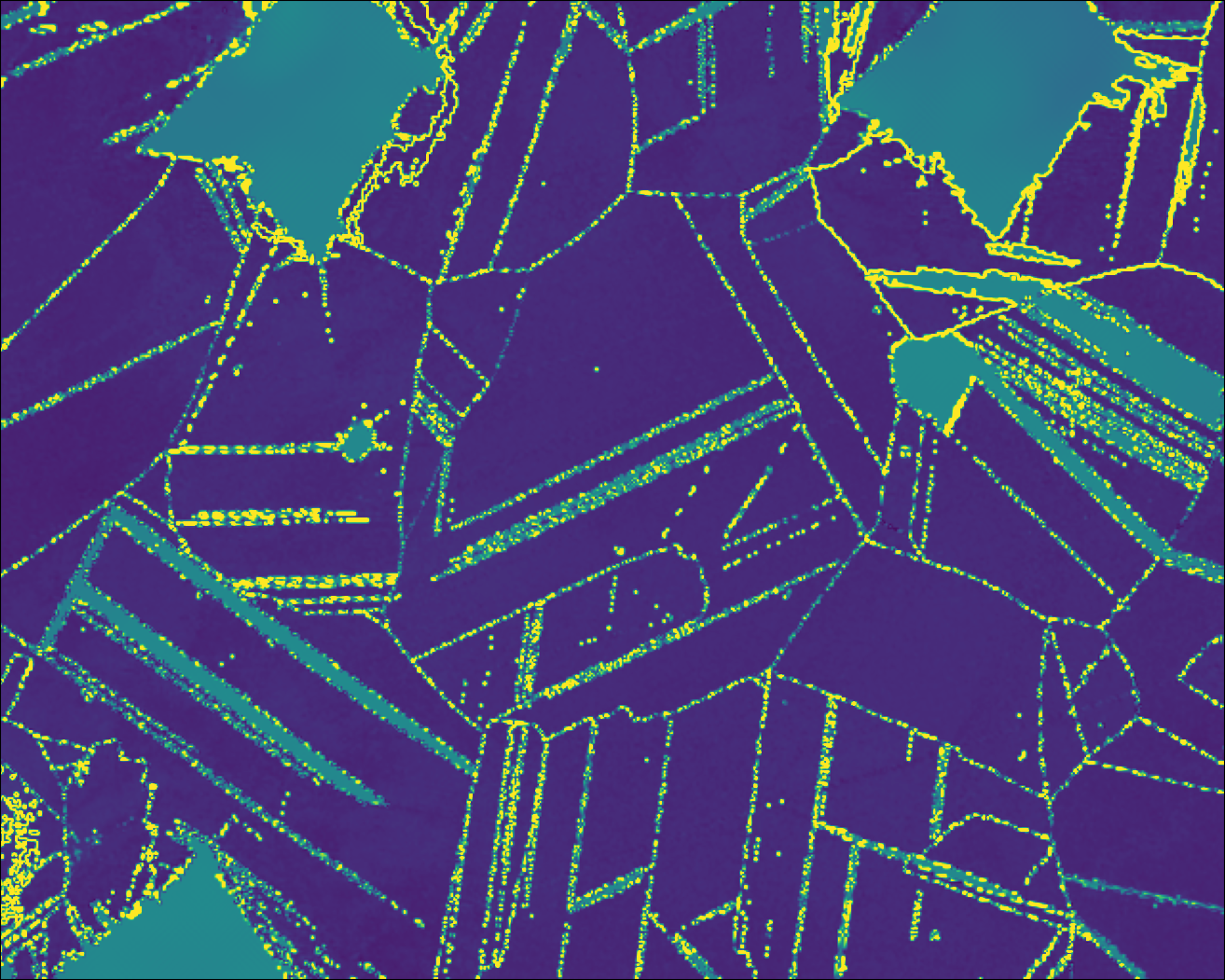}
    \includegraphics[width=0.058\textwidth]{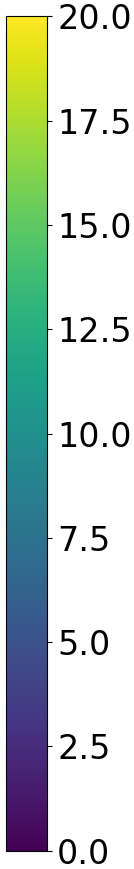}
    \includegraphics[width=0.45\textwidth]{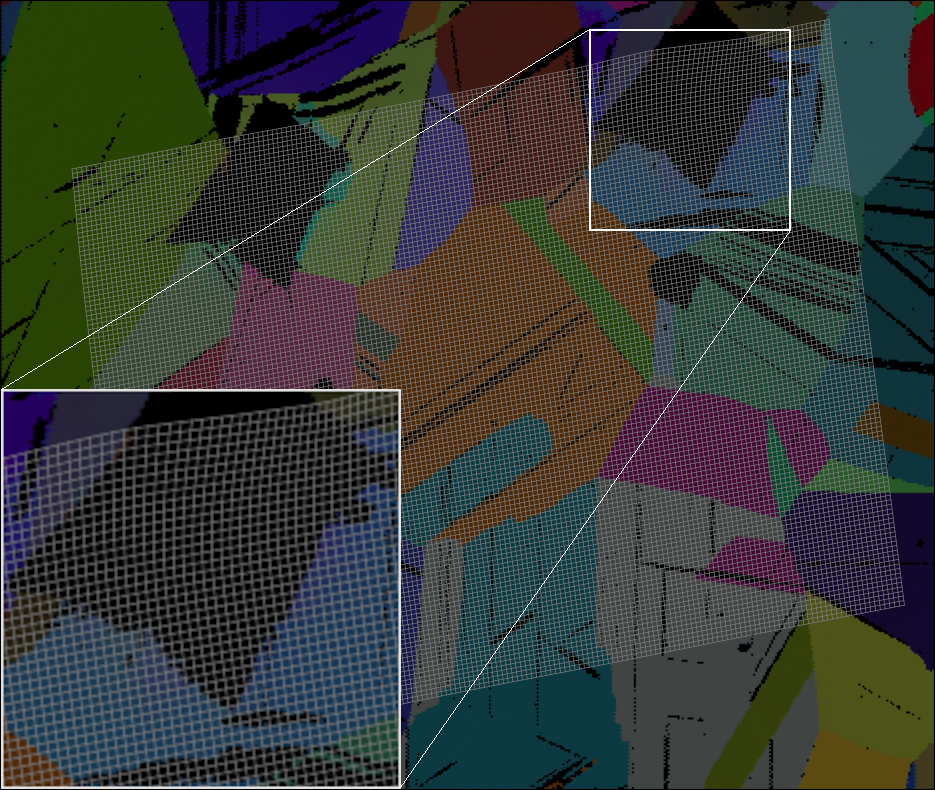}
    \includegraphics[width=0.475\textwidth]{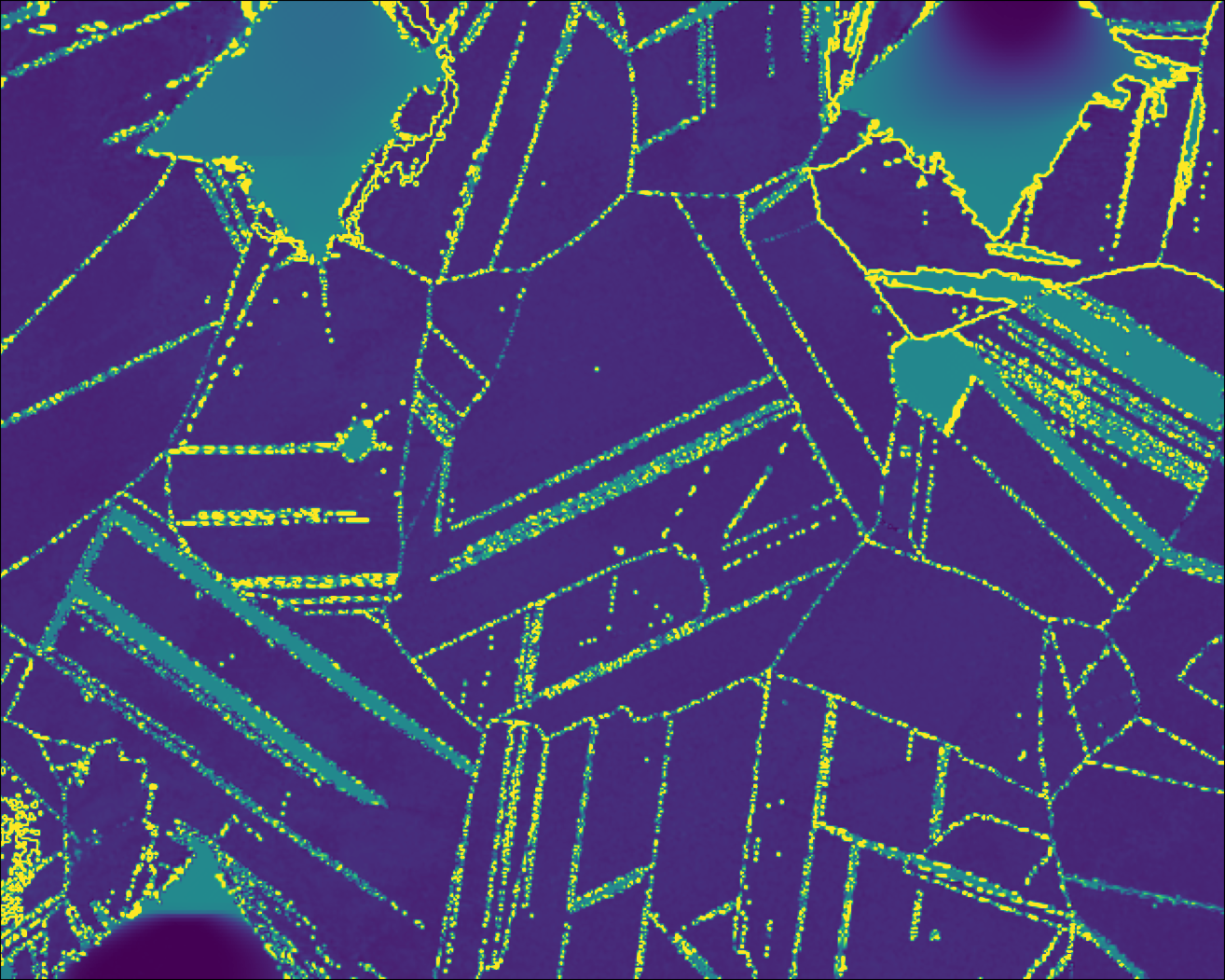}
    \includegraphics[width=0.058\textwidth]{images/colorbar_0_20.png}
    \includegraphics[width=0.45\textwidth]{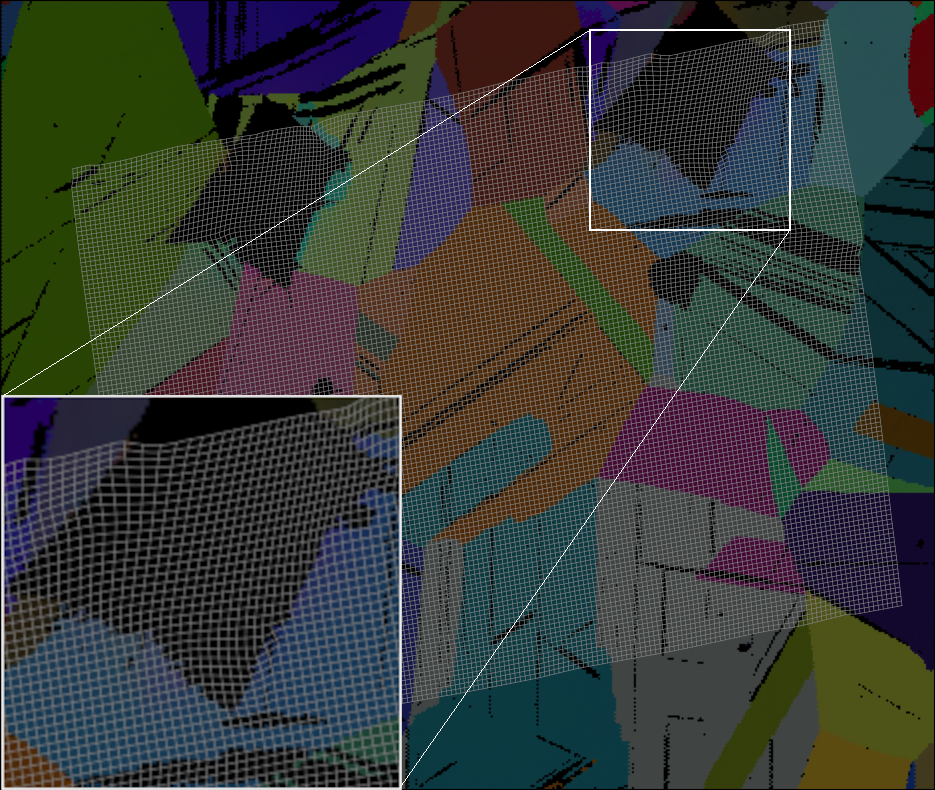}
    \includegraphics[width=0.475\textwidth]{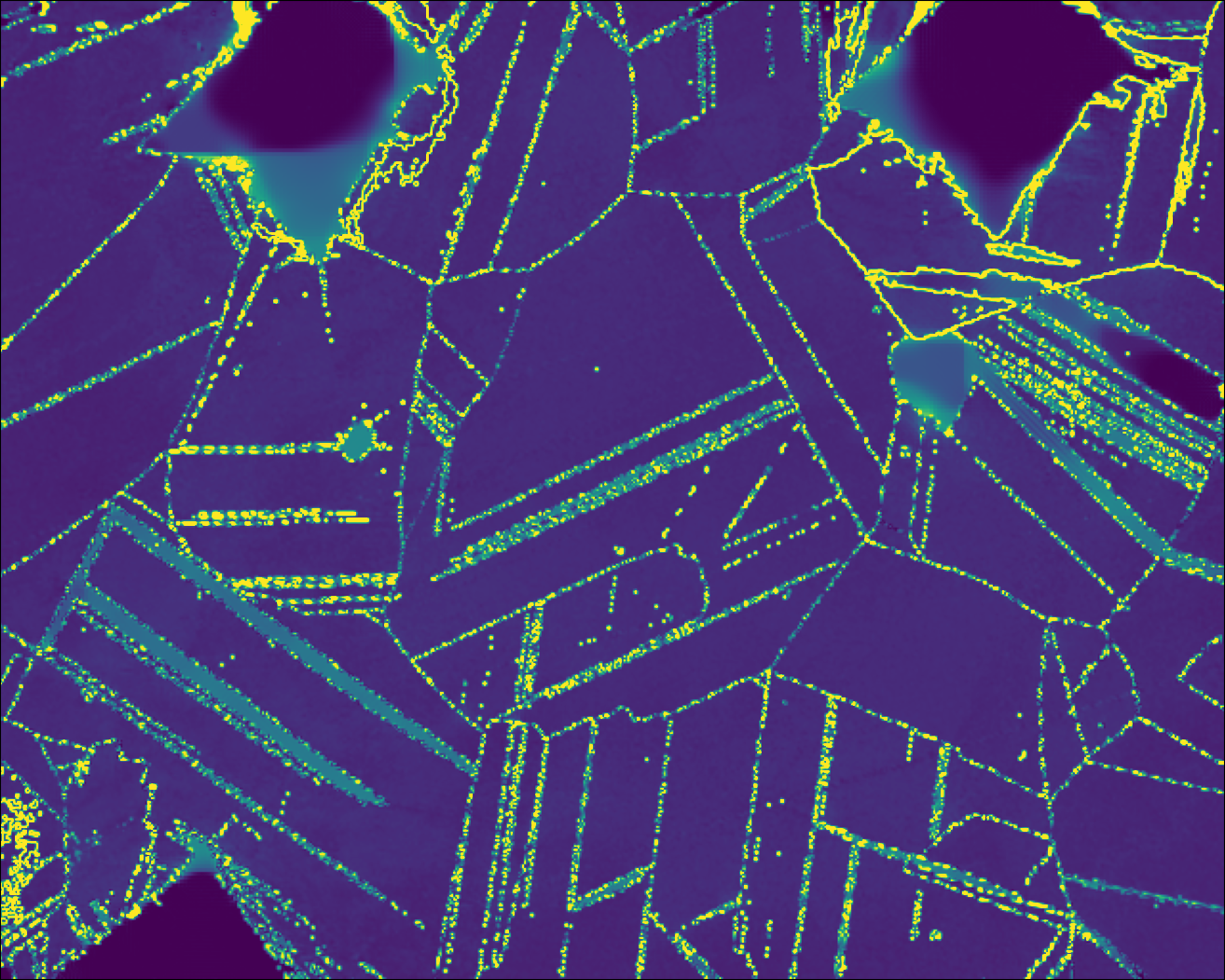}
    \includegraphics[width=0.058\textwidth]{images/colorbar_0_20.png}
    \end{center}
    \caption{
    Real EBSD data of TRIP steel X20MnAlSi16-1-1 measured in two different positions using the austenitic phase. The remaining phases and the indenter imprints are filled by a constant orientation (black regions). Left column: Reference image $I_2$ overlaid with the reconstructed transformation (gray grid) for decreasing regularization parameters $\alpha = 10, 1, 0.1$ (from top to bottom). For every image the part in the upper right is magnified and depicted in the corresponding lower left corner. Right column: The corresponding reconstruction errors in degrees.}
    \label{fig:rotatedSample}
\end{figure}

\subsection{Reconstruction of Simulated Deformations} \label{sec:simulated_deformation}
Finally, we reconstruct the displacement field obtained from a stress simulation of ice crystals
using our $\TV^2$-model. 
More precisely, we are given an initial image $I_1$ 
and 
a displacement field $u$, which has been computed by a physical model, see \cite{LeRo20,LlGrBoRoLeEvWe16}. 
Then, image $I_2$ is obtained by applying the transformation $\varphi(x) = x + u(x)$ to image $I_1$. In the simulation model, the displacement field and the image $I_1$ are continued periodically along the boundary. 
Hence, we need to repeat the image $I_2$ along each direction to account for the periodization. 
The symmetry group of the EBSD data is '6/mmm'.

The input image $I_1$ and the transformed image $I_2$ are depicted in the top row of  Figure~\ref{fig:iceImages}. 
In the bottom row of Figure~\ref{fig:iceImages}, 
we show the result of our method in comparison to  the ground truth given by the transformation $\varphi$.
There is a very good alignment between the given and the reconstructed transformation.

\begin{figure}
  \begin{center}
    \includegraphics[width=0.49\textwidth]{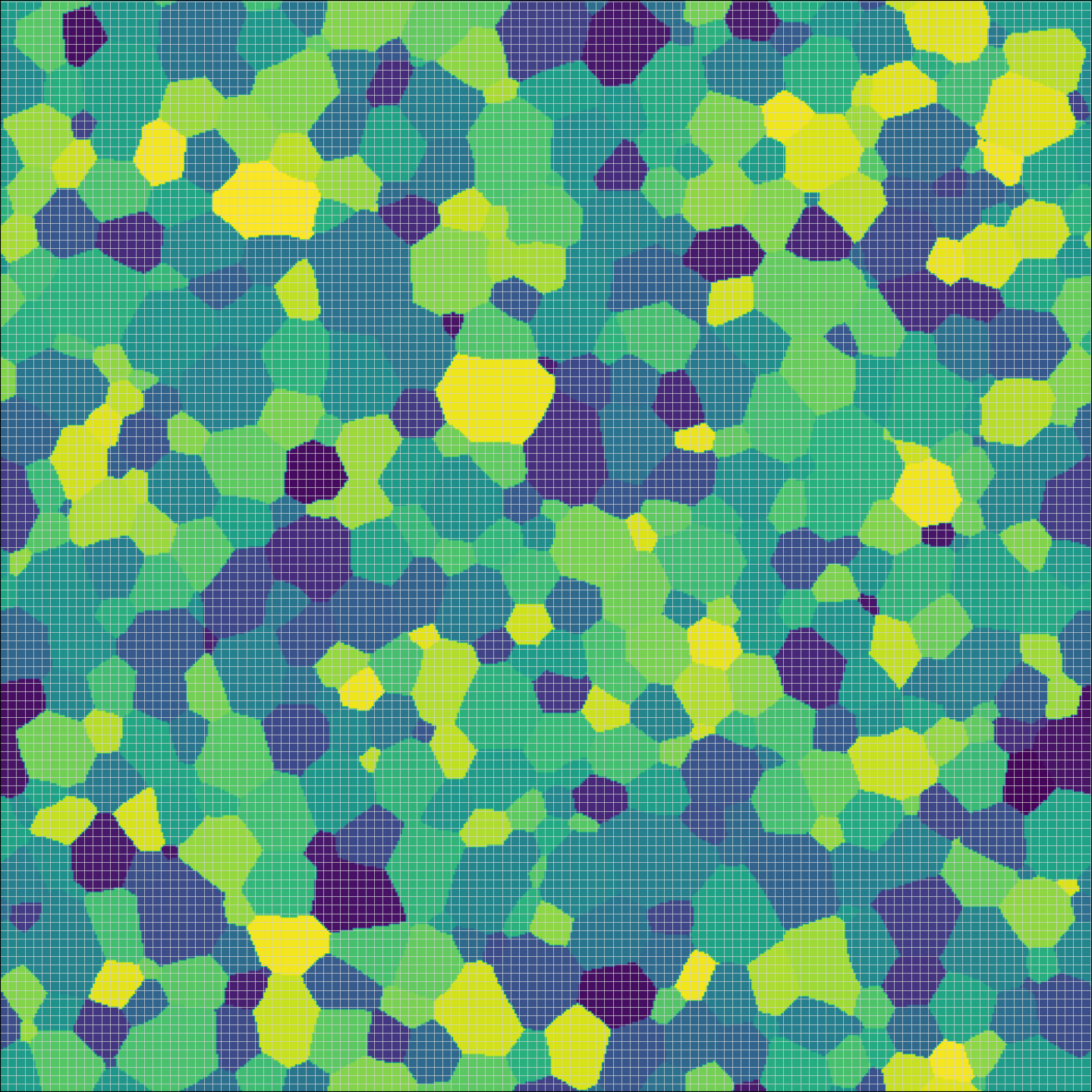}
    \includegraphics[width=0.49\textwidth]{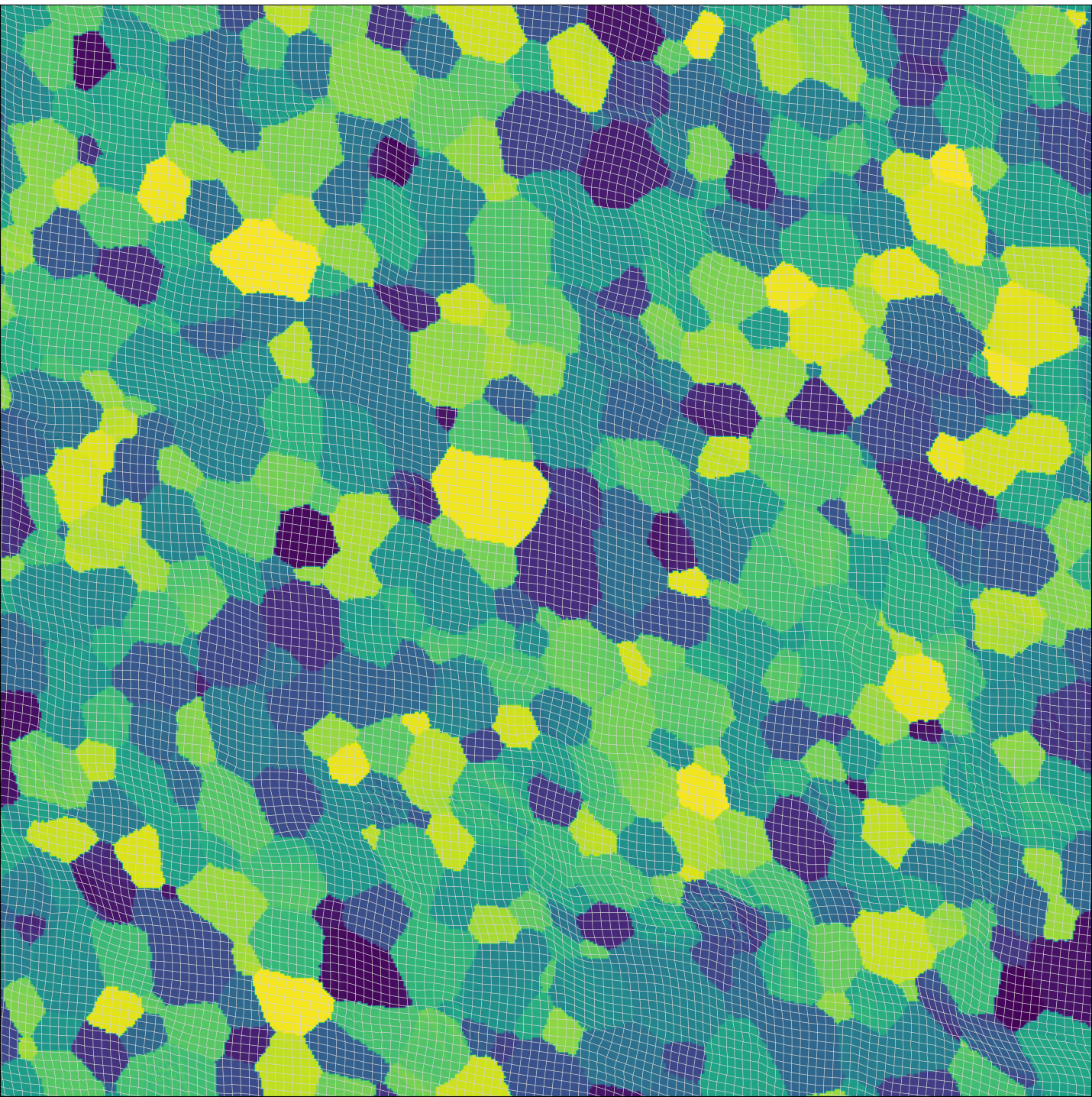}      
    \includegraphics[width=0.49\textwidth]{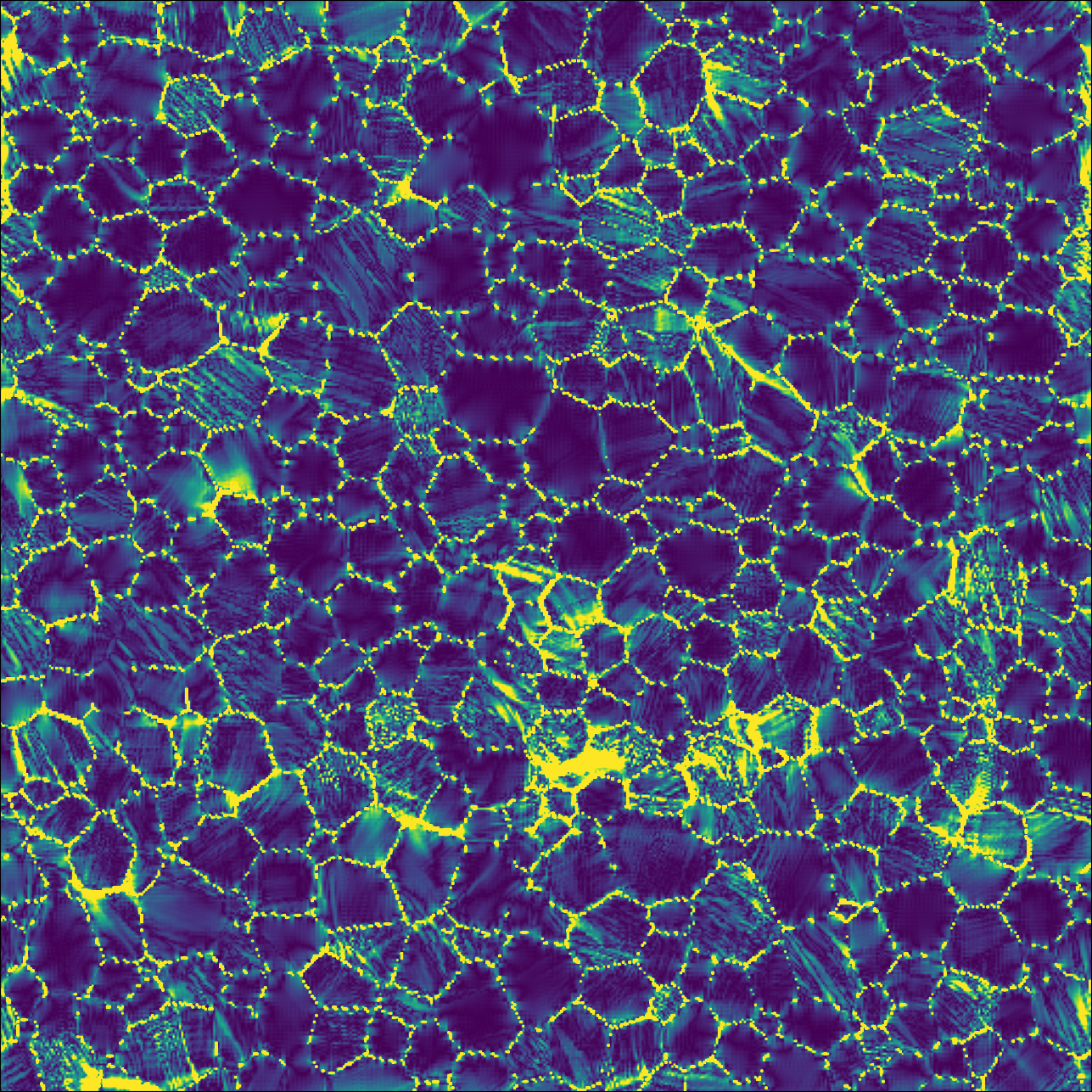}
    \includegraphics[width=0.059\textwidth]{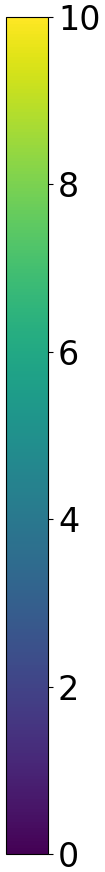}
    \includegraphics[width=0.42\textwidth]{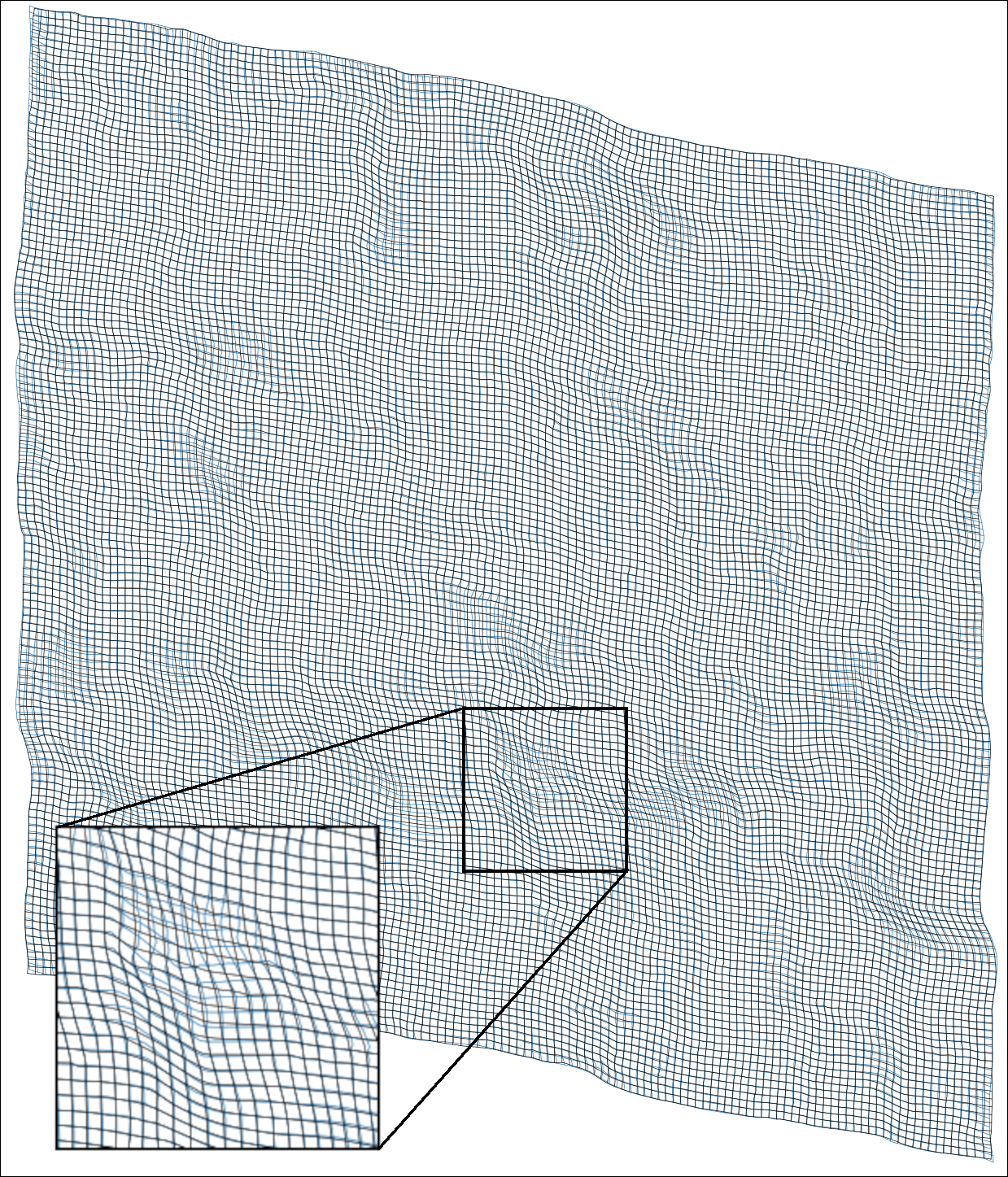}
  \end{center}
  \caption{
  Simulated EBSD data from a stress simulation of ice crystals.
  Top row: Left image $I_1$ is transformed into the right image $I_2$ by shear stress using periodic boundary conditions.  Left border is fixed and right border slips downwards. The corresponding grids are depicted in gray. Bottom left: The reconstruction error in degrees. Bottom right: The original transformation (blue grid) and the reconstructed transformation (black grid). The part in the lower middle is magnified and shows a significant difference to the ground truth.}
  \label{fig:iceImages}
\end{figure}

\section{Conclusions} \label{sec:conclusions}
Nowadays, novel image acquisition techniques such as EBSD come along with challenging tasks in imaging sciences.
In this paper, we were concerned with the appropriate modeling of displacement fields between EBSD image sequences.
In contrast to the gray-value constancy assumption, it appears that the rotation part in the transformation influences the data itself and must be incorporated into the variational model.
Hence, we established a novel continuous non-convex model and took  care in describing its discretization and minimization to make the numerical part understandable and reproducible.

In our future work, we will further refine the methods initialized in this paper.
From a practical perspective, we will be mainly interested in crack detection in fatigue tests and the investigation of real-world materials.
Integrating mathematical models in these investigations could lead to a better understanding of material behavior under mechanical loadings.
In particular, high performance materials such as TRIP steel or anisotropic nickel based superalloys can have a complex microstructure and texture and hardly predictable fatigue failure \cite{Engel.2018,Shan.2008}.
For such materials, EBSD-data is also used to perform PDE-simulations of the deformation behavior of real microstrucures and textures in order to predict the crack initiation and therefore to enhance the performance of components and to save resources.
As the boundary conditions and the mesh generation for real data are challenging and potentially faulty, image registration methods may be a powerful method to verify the simulations by comparing simulated (PDE) 
and reconstructed (image registration) local deformation.
We will also have a look at other deformation models, e.g., metamorphosis related ones.
Finally, we intend to combine image sequences from EBSD and raster electron microscopy
to improve the transformation detection.

\paragraph{Acknowledgement.}
M.G.\ and G.S.\ gratefully acknowledge funding by the German Research Foundation (DFG) 
within the project STE 571/16-1. 
We kindly thank Maria-Gema Llorens for providing us with realistic deformation maps using the Elle Numerical Simulation Platform, Ben Britton and Ruth Birch for measuring the Zirconium data and Stefan Wolke for the EBSD measurements on TRIP steel.

\appendix
\section{Proof of Theorem \ref{thm:basic} } \label{app:A}

The proof of Theorem \ref{thm:basic} is a consequence of the following two lemmata.

\begin{lemma}\label{lem:weak*semi}
The functional $\mathcal R_{\mathrm{relax}}\colon  \BV(\Omega, \R^2) \to \R_{\geq 0} \cup\{+\infty\}$ 
is lower semi-continuous w.r.t.\ $L^1(\Omega, \R^2)$-convergence.
\end{lemma}

\begin{proof}
Let $\varphi_n \to \varphi$ in $L^1(\Omega, \R^2)$.
Without loss of generality, we can assume 
$\mathcal R_{\mathrm{relax}}(\varphi_n) < \infty$ for all $n \in \N$.
By definition of 
$\mathcal R_{\mathrm{relax}}$, there exists $\tilde \varphi_n \in W^{1,2}(\Omega, \R^2)$ 
with $\Vert \varphi_n - \tilde \varphi_n \Vert_1 \leq 1/n$ 
and $\mathcal R(\tilde \varphi_n) \leq \mathcal R_{\mathrm{relax}}(\varphi_n) + 1/n $.
Incorporating $\|\tilde \varphi_n - \varphi\|_1 \le \| \tilde \varphi_n - \varphi_n \|_1 + \|\varphi_n - \varphi\|_1$, 
we obtain $\tilde \varphi_n \to \varphi$ in $L^1(\Omega, \mathbb R^2)$.
Hence, using the definition of $\mathcal R_{\mathrm{relax}}$, we get
\[
\mathcal R_{\mathrm{relax}}(\varphi) 
\leq 
\liminf_{n \to \infty} \mathcal R(\tilde \varphi_n) 
\leq 
\liminf_{n \to \infty} \mathcal R_{\mathrm{relax}}(\varphi_n) + 1/n
= 
\liminf_{n \to \infty} \mathcal R_{\mathrm{relax}}(\varphi_n).
\]
This concludes the proof.  
\end{proof}

Next, we want to show for specific choices of $f$ that 
$\mathcal R_{\mathrm{relax}}(\varphi) \geq \mathcal R(\varphi)$ with equality if $\varphi \in W^{1,2}(\Omega, \R^2)$.
This relation is actually crucial for obtaining coercivity of $\mathcal R_{\mathrm{relax}}$
w.r.t.\ the $\BV$-norm, 
which is one of the main ingredients for applying the direct method of calculus to \eqref{eq:var_problem_orig}.
Here, the following result turns out to be useful.

\begin{theorem}[{\cite[Thm.~3.1]{CD94}}]\label{thm:conv}
  Let $f \colon \R \to \R_{\geq 0} \cup \{+\infty\}$ be convex and lower semi-continuous with $f(0) < \infty$.
  Assume that $\varphi \in \BV(\Omega, \R^2)$ and $\varphi_n \in W^{1,2}(\Omega, \R^2)$ satisfy
  \begin{itemize}
      \item[i)] $(\varphi_n)_n$ is bounded in $W^{1,1}(\Omega, \R^2)$,
      \item[ii)] $\varphi_n \to \varphi$ in $L^{1}(\Omega, \R^2)$
  \end{itemize}
  Then, $\int_\Omega f(\det \nabla \varphi ) \dx x \leq  \liminf_{n \to \infty} \int_\Omega f(\det \nabla \varphi_n ) \dx x$. 
\end{theorem}

Now, we can prove the desired estimate.

\begin{lemma} \label{lem:xxx}
Let $\mathcal R$ and $\mathcal R_{\mathrm{relax}}$ be defined by \eqref{eq:orig} and \eqref{eq:relax}, respectively. 
Then, $\mathcal R_{\mathrm{relax}}(\varphi) \geq \mathcal R(\varphi)$ 
for any $\varphi \in \BV(\Omega, \R^2)$ 
with equality if 
$\varphi \in W^{1,2}(\Omega, \R^2)$.
\end{lemma}

\begin{proof}
First, note that Theorem~\ref{thm:conv} still holds for our choice of $f$.
This can be seen by using monotone increasing approximations 
$f_\eps \colon \R \to \R_{\geq 0} \cup \{+\infty\}$ with 
$f_\eps(x) \coloneqq (x+\eps)^{-1} + x$ if $x\geq 0$ and 
$f_\eps(x) \coloneqq +\infty$ else.
For any sequence $(\varphi_n)_n$ satisfying the conditions of Theorem~\ref{thm:conv}, 
the monotone convergence theorem implies 
\begin{align}
    \int_\Omega f(\det \nabla \varphi ) \dx x 
		&= \lim_{\eps \to 0} \int_\Omega f_\eps(\det \nabla \varphi ) \dx x\notag\\
    &\leq \lim_{\eps \to 0} \liminf_{n \to \infty} \int_\Omega f_\eps(\det \nabla \varphi_n ) \dx x \leq \liminf_{n \to \infty} \int_\Omega f(\det \nabla \varphi_n) \dx x.\label{eq:lsc_one_over_x}
\end{align}
Let $\varphi \in \BV(\Omega, \R^2)$ be arbitrary fixed.
Clearly, we can assume 
$\mathcal R_{\mathrm{relax}}(\varphi)<\infty$, 
otherwise the statement is clear.
Hence, for every $\epsilon \geq 0$,
there exists a sequence $(\varphi_n)_n$ in $W^{1,2}(\Omega, \R^2)$ with $\varphi_n \to \varphi$ in $L^1(\Omega, \R^2)$ and
\[
\limsup_{n \to \infty} \TGV_\alpha(\varphi_n - \mathrm{Id}) 
\le \lim_{n \to \infty} \mathcal R(\varphi_n) 
\leq \mathcal R_{\mathrm{relax}}(\varphi)+ \eps.
\]
As $\TV$ can be upper bounded by $\TGV_\alpha$, see \cite[Cor.~3.13]{BH14}, 
this directly implies that the sequence $(\varphi_n)_n$ 
is bounded in $W^{1,1}(\Omega, \R^2)$.
Using observation \eqref{eq:lsc_one_over_x}, 
i.e., the generalization of Theorem~\ref{thm:conv}, 
and the lower semi-continuity of $\TGV_\alpha$ w.r.t.\ $L^1(\Omega, \mathbb R^2)$-convergence, 
see \cite[Proof of Prop.~3.5]{BKP10}, we get
\[
\begin{aligned}
\mathcal R_{\mathrm{relax}}(\varphi) 
+ \varepsilon \ge  \lim_{n \to \infty} \mathcal R(\varphi_n)
& = \lim_{k\to \infty} \TGV_{\alpha} (\varphi_n - \mathrm{Id}) 
+ \beta \int_{\Omega} f(\det \nabla \varphi_n) \dx x  \\
& \ge \liminf_{k\to \infty} \TGV_{\alpha} (\varphi_n - \mathrm{Id}) 
+ \beta \liminf_{k\to \infty} \int_{\Omega} f(\det \nabla \varphi_n) \dx x \\
& \ge \TGV_{\alpha} (\varphi - \mathrm{Id}) 
+ \beta \int_{\Omega} f(\det \nabla \varphi) \dx x \\
& = \mathcal R(\varphi). 
\end{aligned}
\]
Since $\varepsilon > 0$ was arbitrary, we arrive at $\mathcal R_{\mathrm{relax}}(\varphi) \ge \mathcal R(\varphi)$.
Finally, equality for $\varphi \in W^{1,2}(\Omega, \R^2)$ follows directly 
by choosing 
the constant sequence $\varphi_n \coloneqq \varphi$ 
in the definition of $\mathcal R_{\mathrm{relax}}$.
\end{proof}

Based on the previous lemmata, we can establish the proof using the direct method of calculus.
\\[1ex]
\emph{Proof of Theorem \ref{thm:basic}.}
Let $\varphi_n \in \BV(\Omega, \overline{\Omega}_1)$ be a minimizing sequence.
Then, it holds $\sup_{n}\|\varphi_n \|_1 < \infty$ as $\varphi_n(\Omega) \subset \overline{\Omega}_1$ 
is bounded.
Moreover, \cite[Cor.~3.13]{BH14}, $\mathcal R_{\mathrm{relax}}(\varphi) \ge \mathcal R(\varphi)$ and the triangle inequality imply for any 
$\varphi \in \BV(\Omega, \overline{\Omega}_1)$ that
\begin{align}
    \TV(\varphi) \le C \bigl(\Vert \varphi \Vert_1 + \TGV_\alpha(\varphi)\bigr) \leq C \mathcal R_{\mathrm{relax}}(\varphi) + C.
\end{align}
Hence, we get
\[\sup_{n}\|\varphi_n \|_1 + \TV (\varphi_n) < \infty,\]
and there exists a subsequence converging weakly* to some $\varphi$ in $\BV(\Omega, \overline{\Omega}_1)$.
As weak* convergence implies $L^1$ convergence, Lemma~\ref{lem:weak*semi} implies that the regularizer is weakly* lower semi-continuous.
Hence, the complete functional \eqref{eq:var_problem_orig} is weak* lower semi-continuous and the result follows.

\section{Algorithms } \label{app:B}

\begin{algorithm}[t]
	\caption{\textbf{$u_s$-Minimization}}
	\begin{algorithmic}
		\State \textbf{Parameters:} maximal iterations $r_{\max} \in \mathbb N$, 
		weights $\alpha_1, \alpha_2, \beta, \mu$
		\State \textbf{Input:} initial guess $u_s^{0}$, 
		primal variables $z_s, w_s, \omega_s$, dual variables $\lambda_h,\lambda_g$
    \For{$r \coloneqq 0,\dots,r_{\max}-1$}
      \For{$ l_1 \coloneqq 0,1$}
        \For{$ l_2 \coloneqq 0,1$}
          \State $l \coloneqq 2l_1 + l_2$
          \State $g_s^{r+\frac{l}{4}} \nabla_{u_s} L_\mu(u_s^{r+\frac{l}{4}}, w_s, z_s, \omega_s, \lambda_h, \lambda_g)$
          \For{$ i\coloneqq0, \dots, a/s$}
            \For{$ j\coloneqq0, \dots, b/s$}    
          \State $
            d_{s,i,j}^{r+\frac{l}{4}} \coloneqq 
            \begin{cases}
              - g_{s,i,j}^{r+\frac{l}{4}}, & i=2s + l_1,\; j=2s + l_2,\; \\
              0, &
            \end{cases}$        
            \State $u_{s,i,j}^{r+\frac{l+1}{4}} \coloneqq u_{s,i,j}^{r+\frac{l}{4}} + \tau_{i,j} d_{s,i,j}^{r+\frac{l}{4}}$ \qquad(determine $\tau_{i,j}$ by Algorithm~\ref{alg:LineSearch})
            \EndFor
          \EndFor
        \EndFor
      \EndFor
    \EndFor
		\State \textbf{Output:} approximate minimizer $u_s^{r_{\max}}$
	\end{algorithmic}
	\label{alg:Umin}
\end{algorithm}

\begin{algorithm}[t]
	\caption{\textbf{$w_s$-Minimization}}
	\begin{algorithmic}
		\State \textbf{Parameters:} maximal iterations $r_{\max} \in \mathbb N$, weights $\alpha_1, \alpha_2, \beta, \mu$
		\State \textbf{Input:} initial guess $w_s^{0}$, 
		primal variables $u_s, z_s, \omega_s$, 
		dual variables $\lambda_h,\lambda_g$, 
    \For{$r \coloneqq 0,\dots,r_{\max}-1$}
      \For{$ p \coloneqq 0,1$}
        \State $g_s^{r+\frac{p}{2}} \coloneqq \nabla_{w_s} L_\mu(u_s, w_s^{r+\frac{p}{2}}, z_s, \omega_s, \lambda_h,\lambda_g)$
          \For{$ i\coloneqq 0, \dots, a/s - 1$}
            \For{$ j\coloneqq 0, \dots, b/s - 1$}    
            \State $
            d_{s,i,j}^{r+\frac{p}{2}} \coloneqq 
            \begin{cases}
              - g_{s,i,j}^{r+\frac{p}{2}}, & (i+j) \mod 2 \equiv p,\\
              0, &
            \end{cases}$        
            \State $w_{s,i,j}^{r+\frac{p+1}{2}} \coloneqq w_{s,i,j}^{r+\frac{p}{2}} + \tau_{i,j} d_{s,i,j}^{r+\frac{p}{2}}$ 
						\qquad(determine $\tau_{i,j}$ by Algorithm~\ref{alg:LineSearch})
            \EndFor
          \EndFor
        \EndFor
    \EndFor
		\State \textbf{Output:} approximate minimizer $w_s^{r_{\max}}$
	\end{algorithmic}
	\label{alg:Vmin}
\end{algorithm}

\begin{algorithm}[t]
	\caption{\textbf{$z_s$-Minimization}}
	\begin{algorithmic}
		\State \textbf{Parameters:} maximal iterations $r_{\max} \in \mathbb N$, weights $\alpha_1, \alpha_2, \beta, \mu$
		\State \textbf{Input:} initial guess $z_s^{0}$, 
		primal variables $u_s, w_s, \omega_s$, 
		dual variables $\lambda_h,\lambda_g$, 
    \For{$r \coloneqq 0,\dots,r_{\max}-1$}
      \For{$ i\coloneqq0, \dots,  a/s - 1$}
        \For{$ j\coloneqq0, \dots,  b/s - 1$}
        \State $d_{s,i,j}^{r+\frac12} \coloneqq - \nabla_{z_{s,i,j}}^{\perp} L_\mu(u_s, w_s, z_s^{r}, \omega_s, \lambda_h,\lambda_g)$ \quad ``orthogonal'' $z_s$-gradient \eqref{eq:zsGradOrtho})
        \State $z_{s,i,j}^{r+\frac{1}{2}} 
				\coloneqq 
				z_{s,i,j}^{r} + \tau_{i,j} d_{s,i,j}^{r+\frac{1}{2}}$ \qquad (determine $\tau_{i,j}$ by Algorithm~\ref{alg:LineSearch})
				\State $d_{s,i,j}^{r+1} \coloneqq -  \nabla_{z_{s,i,j}}^{\parallel} 
				L_\mu(u_s, w_s, z_s^{r+\frac12}, \omega_s, \lambda_h,\lambda_g)$ 
				\quad (``parallel'' $z_s$-gradient \eqref{eq:zsGradPar})
        \State $z_{s,i,j}^{r+1} 
				\coloneqq z_{s,i,j}^{r+\frac12} + \tau_{i,j} d_{s,i,j}^{r+1}$
				\qquad(determine $\tau_{i,j}$ by Algorithm~\ref{alg:LineSearch})       
        \EndFor
      \EndFor
    \EndFor
		\State \textbf{Output:} approximate minimizer $z_s^{r_{\max}}$
	\end{algorithmic}
	\label{alg:Zmin}
\end{algorithm}

\begin{algorithm}[t]
	\caption{\textbf{Line Search - Quadratic Interpolation}}
	\begin{algorithmic}
		\State \textbf{Parameters:} maximal iterations $r_{\max} \in \mathbb N$, 
		quadratic interpolation step $h > 0$, maximal step $m>0$, 
		step size factor $0<\sigma<1$
    \State \textbf{Input:} base point $x \in \mathbb R^d$, descent direction 
		$d \in \mathbb R^d$, function $f\colon\mathbb R^d \to \mathbb R$
    \State  $a \coloneqq (f(x - h d) - 2 f(x) + f(x + h d)) / h^2$ 
    \State  $b \coloneqq (f(x - h d) - f(x + h d)) / (2h)$
    \State (interpolation coefficients of quadratic function $q(t) = \frac12 a t^2 + b t + c$)
    \If{$a \ne 0$}
    \State $\tau^0 \coloneqq \min( \frac{|b|}{|a|}, m)$ \quad (initial step size from quadratic approximation)
    \Else
      \State $\tau^0 = 0$ 
    \EndIf
    \State $r\coloneqq0$
    \While{$f(x + \tau^0 d) \ge f(x)$ \textbf{and} $r<r_{\max}$}
    \State $\tau^{r+1} \coloneqq \sigma \tau^r$
    \State $r \coloneqq r + 1$ 
    \EndWhile
    \If{$r =r_{\max}$}
      \State $\tau^r \coloneqq 0$ \quad (line search failed)
    \EndIf
    \State \textbf{Output:} step size $\tau_{r}$
	\end{algorithmic}
	\label{alg:LineSearch}
\end{algorithm}

\bibliographystyle{abbrv}
\bibliography{references,refs}

\end{document}